\documentclass[12pt,oneside,reqno]{amsart}

\usepackage{color}
\usepackage{soul}
\usepackage{mathtools} 
\usepackage{graphicx}
\usepackage{bbm}
\usepackage{physics}
\usepackage{enumitem}
\usepackage{tikz}
\usepackage{amssymb,amsthm,amsmath, subcaption}
\usepackage[backref=page, colorlinks=true, allcolors=blue]{hyperref}
\usepackage[abbrev,lite,backrefs]{amsrefs}
\usepackage{color}

\usepackage{graphicx}
\usetikzlibrary{shapes}
\usepackage{xparse}
\usepackage{bm}
\usepackage{array}
\usetikzlibrary{calc,arrows.meta,positioning}

\usepackage[colorinlistoftodos,prependcaption,textsize=scriptsize, obeyFinal]{todonotes}

\makeatletter
\providecommand\@dotsep{5}
\renewcommand{\listoftodos}[1][\@todonotes@todolistname]{%
  \@starttoc{tdo}{#1}}
\makeatother

\makeatletter
\def\Ddots{\mathinner{\mkern1mu\raise\p@
\vbox{\kern7\p@\hbox{.}}\mkern2mu
\raise4\p@\hbox{.}\mkern2mu\raise7\p@\hbox{.}\mkern1mu}}
\makeatother

\AtBeginEnvironment{example}{%
  \pushQED{\qed}%
}
\AtEndEnvironment{example}{\popQED\endexample}

\title[Inverse Problems, Floquet Isospectrality and the Hilbert--Chow Morphism]{Inverse Eigenvalue Problems, Floquet Isospectrality and the Hilbert--Chow Morphism}
\author[J. Cobb]{John Cobb}
\address{John Cobb, Department of Mathematics and Statistics, Auburn University, Auburn, AL}
\email{jdcobb3@gmail.com}
\urladdr{\href{https://johndcobb.github.io}{https://johndcobb.github.io}}

\author[M. Faust]{Matthew Faust}
\address{Matthew Faust, Department of Mathematics, Michigan State University, East Lansing, MI 48824, USA} \email{mfaust@msu.edu}
\urladdr{\href{https://mattfaust.github.io}{https://mattfaust.github.io}}

\author[A. Kretschmer]{Andreas Kretschmer}
\address{Andreas Kretschmer, Institut für Mathematik, Humboldt-Universität zu Berlin, Berlin, Germany} \email{andreas.kretschmer@hu-berlin.de}
\urladdr{\href{https://sites.google.com/view/andreas-kretschmer/}{https://sites.google.com/view/andreas-kretschmer/}}


\theoremstyle{plain}
\newtheorem{theorem}{Theorem}[section]
\newcommand{\R}{\mathbb{R}}
\newtheorem{corollary}[theorem]{Corollary}
\newtheorem{lemma}[theorem]{Lemma}
\newtheorem{defn}[theorem]{Definition}
\newtheorem{proposition}[theorem]{Proposition}

\newcommand{\CC}{\mathbb{C}}
\newcommand{\TT}{\mathbb{T}}
\newcommand{\ZZ}{\mathbb{Z}}
\newcommand{\NN}{\mathbb{N}}
\renewcommand{\AA}{\mathbb{A}}

\newcommand{\bfz}{\bf{0}}
\usepackage{xparse}
\usetikzlibrary{calc,arrows.meta,positioning}

\DeclareMathOperator{\mult}{mult}
\DeclareMathOperator{\Spec}{Spec}
\DeclareMathOperator{\Hilb}{Hilb}
\DeclareMathOperator{\Sym}{Sym}
\DeclareMathOperator{\Hom}{Hom}
\DeclareMathOperator{\HC}{HC}

\theoremstyle{plain}
\theoremstyle{definition}
\newtheorem{definition}{Definition}
\newtheorem{conjecture}{Conjecture}
\newtheorem{example}[theorem]{Example}
\newtheorem{rmk}[theorem]{Remark}

\newcommand{\demph}[1]{{\sl #1}}

\begin{document}

	\begin{abstract}
    When can one change the diagonal of a matrix without changing its spectrum? We completely answer this question over an algebraically closed field of characteristic zero or larger than the size of the matrix: An $n \times n$ matrix $A$ admits a nonzero diagonal matrix $D$ such that $A$ and $A+D$ have the same spectrum if and only if, for some size $k$, the $k \times k$ principal minors of $A$ are not all equal. This relates to the classical additive inverse eigenvalue problem in numerical analysis and has implications for existence and rigidity results in the theory of Floquet isospectrality of discrete periodic operators in solid state physics.   The proof employs new techniques involving Hilbert schemes of points and the infinitesimal structure of the Hilbert--Chow morphism.
	\end{abstract}
	
	\maketitle 

    \vspace{-0.25in}

	\section{Introduction}

    \renewcommand*{\thetheorem}{\Alph{theorem}}

    The class of inverse eigenvalue problems is among the most studied in matrix theory and applied analysis, with contributions from nearly every field of mathematics; for a nice survey, see \cite{MoodySurvey}. Let $A$ be a fixed $n \times n$ matrix over an algebraically closed field~$K$. One of the most famous among them is the \textit{additive inverse eigenvalue problem (AIEP)}:
    \begin{equation}\tag{$*$}\label{prob:additive inverse eigenvalue}
        \parbox{\dimexpr\linewidth-4em}{%
        Given a multiset $\{\lambda_1,\dots,\lambda_n\} \subset K$, find a diagonal matrix $D$ such that the spectrum of $A+D$ is $\{\lambda_1,\dots,\lambda_n\}$.
      }
    \end{equation}
    Over $\CC$, Friedland showed that problem \eqref{prob:additive inverse eigenvalue} is always solvable and, for a generic choice of target spectrum, there are exactly $n!$ different solutions \cite{friedland1972matrices}. Other proofs of \eqref{prob:additive inverse eigenvalue} appear in \cites{FRIEDLAND197715, alexander1978additive}, and other variations of the same problem are investigated in \cites{kapi, VPapan}.
    
    Friedland's result can be rephrased to say that every fiber of the morphism
    \begin{equation*}
        F_A: \operatorname{Diag}_n(\CC) \to \CC^n/\Sigma_n,  \hspace{0.3in} D \mapsto \sigma(A+D),
    \end{equation*}
    sending a diagonal matrix $D$ to the spectrum of $A+D$, has $n!$ points when counted with multiplicities. Here, $\Sigma_n$ is the symmetric group on $n$ elements. In algebro-geometric terms, $F_A$ is finite and flat of degree $n!$, and in particular surjective.
    
    In this paper, we study the fibers of $F_A$, especially the fiber over the spectrum of $A$ itself, consisting of the diagonal matrices $D$ such that $A$ and $A+D$ are isospectral. The other fibers are obtained from this one upon replacing $A$ by a diagonal shift. Our main theorem gives an intrinsic characterization of the ``spectrally rigid'' matrices $A$ for which the central fiber is a \emph{singleton}, that is, when there are no diagonal shifts $D \neq 0 $ preserving the spectrum of $A$.

    \begin{theorem}\label{thm:IntroMain}
        Fix an $n\times n$ matrix $A$ with entries in an algebraically closed field of characteristic $0$ or $>n$. There exists a nonzero diagonal matrix $D$ such that $A$ and $A+D$ have the same spectrum if and only if for some $k$, the $k \times k$ principal minors of $A$ are not all equal.
    \end{theorem}

    \noindent Theorem \ref{thm:IntroMain} identifies the set of spectrally rigid matrices with the so-called \textit{variety of symmetrized principal minors} consisting of $n\times n$ matrices whose principal minors of equal size have the same value. This variety has been studied and described by Huang--Oeding \cite{huang2017symmetrization} in connection to the famous principal minor assignment problem \cites{al2024characterizing, holtz2007hyperdeterminantal,grinshpan2013norm}. For instance, a triangular matrix has symmetrized principal minors if and only if all of its diagonal entries are equal. In this case, Theorem \ref{thm:IntroMain} can be verified by direct computation.
    
    When $A$ is the adjacency matrix of a weighted simple graph $G$, Theorem \ref{thm:IntroMain} becomes a new result in spectral graph theory.

    \begin{corollary}\label{cor: adjacency matrix}
    Suppose that $A$ is the adjacency matrix of a weighted simple graph $G$. There exists a nonzero diagonal matrix $D$ such that $A$ and $A+D$ have the same spectrum if $G$ is neither edgeless nor complete with all edge weights equal up to sign. 
\end{corollary}
    \noindent One motivation for Corollary \ref{cor: adjacency matrix} comes from the spectral theory of discrete periodic Schr\"{o}dinger operators. Let $\Delta$ denote the discrete Laplacian on $\ZZ^d$ and let $V: \ZZ^d \to \CC$ be a $\Lambda$-periodic potential, where 
    \begin{equation*}
        \Lambda = q_1\ZZ \oplus \cdots \oplus q_d \ZZ, \qquad q_j \in \ZZ_{>0}.
    \end{equation*}
    
    The operator $\Delta+V$ acts on $\ell^2(\ZZ^d)$ and serves as a standard tight–binding model for electron dynamics in crystalline solids and related materials; see the survey~\cite{ksurvey}. Floquet theory identifies the spectrum of $\Delta+V$ with the union, over $z\in (\TT)^d$, of the spectra of the finite matrices $L_V(z)$ of size $q = \prod_{j=1}^d q_j$ obtained by imposing Floquet boundary conditions, where $\TT$ denotes the complex unit circle. Two such operators $\Delta+V$ and $\Delta+V'$ are said to be \emph{Floquet isospectral} if $\sigma(L_V(z))=\sigma(L_{V'}(z))$ for all $z\in (\TT)^d$, or, equivalently, if they have the same band structure.

    It is a classical rigidity theorem that there are no nonzero \emph{real} potentials Floquet isospectral to the zero potential $\mathbf{0}$, in either the discrete or continuous setting; see, e.g., Borg~\cite{borg} and Kuchment~\cite{kapiii}. In contrast, for complex potentials there are many nontrivial examples isospectral to~$\mathbf{0}$ in the continuous case~\cite{GU}. For discrete periodic operators, on the other hand, such examples have been constructed only in special situations~\cites{kapi,VPapan,flmrp}. By letting $A$ be the discrete Laplacian and $D$ be a periodic multiplication operator, Theorem \ref{thm:IntroMain} yields the following classification result in the discrete case.
    
\begin{theorem} \label{Cor:FloquetMain} 
    Let $\Lambda=q_1\mathbb{Z}\oplus q_2 \mathbb{Z}\oplus\cdots\oplus q_d\mathbb{Z}$  with $q_j\in\ZZ_{>0}$. If at least one of the periods $q_j$ is greater than $3$, then there exists a nonzero $\Lambda$–periodic complex potential $V$ Floquet isospectral to $\mathbf{0}$. If all $q_j\leq 3$ and at most two of the $q_j$ are equal to $3$, then the only $\Lambda$–periodic potential Floquet isospectral to $\mathbf{0}$ is $V\equiv \mathbf{0}$.
\end{theorem}

\noindent This is a new phenomenon in the discrete setting: unlike in the continuous setting, complex periodic potentials remain rigid for sufficiently small periods and fail for complex potentials whose largest period exceeds $3$.

We now briefly outline the proof of Theorem \ref{thm:IntroMain}. For $A$ and $A+D$ to be isospectral, they must have the same  characteristic polynomial. Equivalently, all the coefficients $S_1,\dots, S_n$ of the polynomial $\det(A-\lambda I) - \det(A+D-\lambda I)$, which we call the \textit{spectral invariants} of $A$, must vanish. The common zero locus of the ideal
\[
    E_A \coloneqq (S_1,\dots,S_n)\subset k[v_1,\dots,v_n],
\]
known as the \emph{ideal of spectral invariants}, is precisely the set of diagonal matrices $D$ such that $A+D$ is isospectral to $A$. After expanding out $S_i$ in terms of principal minors, one direction of Theorem~\ref{thm:IntroMain} is relatively direct: if all $k \times k$ principal minors of $A$ agree for every $k$, then we show that $E_A = (e_1,\dots,e_n)$, where $e_i$ denotes the $i$-th elementary symmetric polynomial in the $v_j$. In particular, its vanishing set is $\mathcal{V}(E_A)=\{0\}$.

For the converse direction, we consider the degeneration of $E_A$ to its associated graded ideal, giving a flat family parametrized by $\AA^1$, and study its image under the Hilbert–Chow morphism
\[
    \HC\colon \Hilb_{n!}(\AA^n)\longrightarrow \Sym_{n!}(\AA^n),
\]
encoding how the solutions of the polynomial system $E_A$ move in the degeneration. If there were no nontrivial diagonal shifts $D\neq 0$ preserving the spectrum of $A$, the Hilbert–Chow image of this family would have to remain constant. We use the Grothendieck--Deligne norm map description of $\HC$ in combination with the  combinatorics of the coinvariant algebra $k[v_1,\dots,v_n]/(e_1,\dots,e_n)$ and some information on what we call the Artin reduction of monomials, to show that this is impossible if $A$ lies outside the symmetrized-principal-minor locus.

This article is organized as follows. In Section~\ref{SEC: background}, we provide the necessary background on the coinvariant algebra for the symmetric group, Hilbert schemes of points and the functorial description of the Hilbert--Chow morphism via the Grothendieck--Deligne norm map. Section~\ref{SEC:MainResult} develops how the latter gives a practical tool for establishing the existence of nonzero solutions for the ideals $E_A$ which leads to the proof of Theorem~\ref{thm:IntroMain} and Corollary~\ref{cor: adjacency matrix}. The application to the theory of Floquet isospectrality for discrete periodic Schr\"odinger operators, especially the proof of Theorem~\ref{Cor:FloquetMain}, is given in the final Section~\ref{SEC:Floquet Isospectrality}.

\section{Background and Notation}\label{SEC: background}
\renewcommand*{\thetheorem}{\thesection.\arabic{theorem}}
Let $K$ be a field of characteristic $0$ or $>n$ and let $P = K[v_1,\dots, v_n]$ be a polynomial ring with $n$ variables.

\subsection{Spectral invariants}
Let $A$ be a $n\times n$ matrix over $K$ and define $D\coloneqq \operatorname{diag}(v_1,\dots, v_n)$. Consider the characteristic polynomial of $A+D$ with coefficients in $P$:
\begin{equation}
    \det(A+D-\lambda I) = \sum_{i=0}^n C_i(v)(-\lambda)^{n-i},
\end{equation}
where each $C_i(v)\in P$ is homogeneous of degree $i$ in the $v_j$. If $A$ and $A+D$ are isospectral, then their characteristic polynomials coincide. This motivates the following definition.

\begin{defn}\label{defn:spectralInvariants}
    The \emph{spectral invariants} of $A$ are the polynomials 
    \begin{equation*}
        S_i \coloneqq C_i(v) - C_i(0) \in P, \qquad i = 1, \ldots, n,
    \end{equation*}
    Equivalently, $S_i$ is the $i$-th coefficient of $\det(A-\lambda I) - \det(A+D-\lambda I)$. The ideal
    \begin{equation*}
        (S_1,\dots, S_n) \subset P
    \end{equation*}
    is called the \emph{ideal of spectral invariants} of $A$.
\end{defn}

Let $A_N$ be the principal submatrix of $A$ with respect to the rows and columns indexed by $N \subseteq [n]$. The following formula for the spectral invariants $S_i$ is an easy consequence of the Leibniz expansion for determinants.

\begin{lemma}\label{lem:SpectralInvariants}
For all $1 \leq i \leq n$, we have
    \begin{equation*}
        S_i = \sum_{\substack{J \subseteq [n] \\ 1 \leq |J| \leq i}} \left( \sum_{\substack{N \subseteq [n] \setminus J \\ |N| = i - |J|}} \det(A_N) \right) v^J,
    \end{equation*}
    where $v^J \coloneqq \prod_{j \in J} v_j$.
\end{lemma}

\subsection{The coinvariant algebra}
Let $e_1, \ldots, e_n \in P$ be the elementary symmetric polynomials, i.e., $e_i$ is the sum of all square-free monomials of degree $i$ in $v_1, \ldots, v_n$. 

\begin{defn}
The \emph{complete homogeneous symmetric polynomial} $H_i(v_1,\dots,v_j)$ is the sum of all monomials of total degree $i$ in the variables $v_1,\dots,v_j$.
\end{defn}

The elementary symmetric polynomials and the complete homogeneous symmetric polynomials are related by the following polynomial identity.

\begin{lemma}\label{lemma:CompleteHomogeneous}
    For every $1 \leq j \leq n$ we have the relation
    \begin{equation*}
        \sum_{i = 0}^{n-j+1} (-1)^i H_i(v_1, \ldots, v_j) e_{n-j+1-i}(v_1, \ldots, v_n) = 0.
    \end{equation*}
\end{lemma}

\begin{proof}
    The generating series for the complete homogeneous symmetric polynomials in $v_1, \ldots, v_j$ is
    \begin{equation*}
        \sum_{i \geq 0} H_i(v_1, \ldots, v_j) t^i = \frac{1}{(1 - v_1 t) \cdots (1 - v_j t)},
    \end{equation*}
    hence
    \begin{equation*}
         \sum_{i \geq 0} (-1)^i H_i(v_1, \ldots, v_j) t^i = \frac{1}{(1 + v_1 t) \cdots (1 + v_j t)}.
    \end{equation*}
    On the other hand, the generating series for the elementary symmetric functions in $v_1, \ldots, v_n$ is
    \begin{equation*}
        \sum_{i = 0}^n e_{i}(v_1, \ldots, v_n) t^{i} = (1 + v_1 t) \cdots (1 + v_n t).
    \end{equation*}
    The coefficient of $t^{n-j+1}$ of the product of the latter two generating series is then equal to the sum in the statement of the lemma. On the other hand, the product of these two series is $(1 + v_{j+1} t) \cdots (1 + v_n t)$ (to be interpreted as $1$ for $j = n$). Expanding this product, the highest possible exponent of $t$ is $n-j$, so the coefficient of $t^{n-j+1}$ equals zero.
\end{proof}

\begin{corollary}\label{cor:v1^q}
   $v_1^n = \sum_{i = 1}^n (-1)^{i-1} v_1^{n-i} e_i$.
\end{corollary}

\begin{proof}
    Apply Lemma~\ref{lemma:CompleteHomogeneous} with $j = 1$.
\end{proof}

A direct consequence of Lemma~\ref{lemma:CompleteHomogeneous} is the equality of ideals
\begin{equation*}
     (e_1, \ldots, e_n) = (H_1(v_1, \ldots, v_n), \ldots, H_{n-1}(v_1,v_2), H_n(v_1)) \subseteq P.
\end{equation*}
Moreover, the $H_i(v_1, \ldots, v_{n-i+1})$ are readily seen to form a Gröbner basis for the lexicographic term order for which $v_1 < v_2 < \cdots < v_n$ with corresponding initial ideal $(v_1^n, v_2^{n-1}, \ldots, v_n)$. It is a consequence of this that the so-called \emph{Artin monomials} $v_1^{l_1} \cdots v_n^{l_n}$ with $0 \leq l_i \leq n-i$ for all $i = 1, \ldots, n$ form a graded $K$-basis for the \emph{coinvariant algebra} $P/(e_1, \ldots, e_n)$.

\begin{defn}
    Let $p \in P$ be any polynomial. The \emph{Artin reduction} of $p$ is the unique linear combination of Artin monomials $p' \in P$ such that $\overline{p} = \overline{p'}$ in $P/(e_1, \ldots, e_n)$.
\end{defn}

The size of the exponents appearing in terms of the Artin reduction of a monomial $p\in P$ are in general hard to predict. However, the following can be said for the exponent of $v_1$.

\begin{lemma}\label{lemma:ArtinMonomialReduction}
    Let $p \in P$ be a monomial and write $p = v_1^r v$ for a monomial $v$ not divisible by $v_1$. Then the exponent of $v_1$ in every Artin monomial occurring in the Artin reduction $p'$ of $p$ is $\geq r$.
\end{lemma}

\begin{proof}
    This can be seen inductively. If $p$ is an Artin monomial, this is trivial. Otherwise, let $i \geq 1$ be maximal such that the exponent $r_i$ of $v_i$ in $p$ satisfies $r_i \geq q-i+1$ and let $p'$ be the Artin reduction of $p$. If $i = 1$, then $p \in (e_1, \ldots, e_q)$ by Corollary~\ref{cor:v1^q}, so $p' = 0$ and the claim is vacuously true. If $i \geq 2$, then $p \equiv p - v_i^{r_i - (q-i+1)} H_{q-i+1}(v_1, \ldots, v_i) \pmod {(e_1, \ldots, e_q)}$ and all monomials in $p - v_i^{r_i - (q-i+1)} H_{q-i+1}(v_1, \ldots, v_i)$ have equal or larger $v_1$-exponents than $p$, strictly lower $v_i$-exponents than $p$, and for $j > i$ the exponents of all the $v_j$ remain the same as in $p$. We may repeat the same procedure with all monomials that occur in $p - v_i^{r_i - (q-i+1)} H_{q-i+1}(v_1, \ldots, v_i)$, discarding those with $v_1$-exponent $\geq q$ as before. In each iteration, for every new monomial either the maximum index $i$ for which the exponent of $v_i$ is $\geq q-i+1$ drops, or $i$ stays the same and the exponent of $v_i$ drops, so this process terminates.
\end{proof}

\subsection{Hilbert schemes of points and the Hilbert--Chow morphism}\label{sec: AG background}
For general background on algebraic geometry and scheme theory we refer to \cite{H}. We briefly introduce Hilbert schemes of points and the Hilbert--Chow morphism in the setting relevant to us here. Much more can be found, e.g., in \cite[Chapters~5-7]{FGA}.

As before, let $P = K[v_1, \ldots, v_n]$ be the polynomial ring. We denote by $\AA^n \coloneqq \Spec(P)$ the affine $n$-space over the algebraically closed field $K$. For an integer $r \geq 1$, the \emph{Hilbert scheme $\Hilb_r(\AA^n)$ of $r$ points on $\AA^n$} parametrizes all zero-dimensional subschemes of $\AA^n$ of length $r$ or, equivalently, all ideals $I \subseteq P$ such that $\dim_K(P/I) = r$. More precisely, the functor
\begin{align*}
    \mathcal{H}ilb_r(\AA^n) : \ &\mathrm{Sch}_{/K}^{\mathrm{op}} \longrightarrow \mathrm{Sets}, \\
    & \quad \ \ T \longmapsto \{Z \subseteq \AA^n \times_K T \text{ closed subscheme } | \ Z \rightarrow T \text{ finite flat of rank } r\}
\end{align*}
is represented by a quasi-projective scheme $\Hilb_r(\AA^n)$, i.e., there are bijective maps $\mathcal{H}ilb_r(\AA^n)(T) \cong \Hom_{\mathrm{Sch}_{/K}}(T, \Hilb_r(\AA^n))$ that are natural in~$T$.

The geometry of $\Hilb_r(\AA^n)$ is subtle in general and is still very actively studied \cites{Iarrobino1972Reducibility, Cartwright2009HilbertSchemeOf8Points, JelisiejewPathologies, JelisiejewGenericallyNonreduced, JelisiejewComponents, Farkas2024Irrational}. Regarding positive results, it is known that $\Hilb_r(\AA^n)$ is connected for all $n$ and $r$~\cite{HartshorneConnectedness}, and that $\Hilb_r(\AA^n)$ is smooth and hence irreducible if $r \leq 3$ \cite[Section~7.2]{FGA} or $n \leq 2$~\cite{Fogarty}.

For the rest of this section we assume that the characteristic of $K$ is zero or does not divide $r$. We will make extensive use of the \emph{Hilbert--Chow morphism}
\begin{equation*}
    \HC: \ \Hilb_r(\AA^n) \longrightarrow \Sym_r(\AA^n).
\end{equation*}
Here, $\Sym_r(\AA^n)$ denotes the $r$-th symmetric power and can be defined as
\begin{equation*}
    (\AA^n)^r/\Sigma_r = \Spec \left( (P^{\otimes r})^{\Sigma_r} \right),
\end{equation*}
where $\Sigma_r$ is the symmetric group permuting the $r$ factors of $(\AA^n)^r$ and $(P^{\otimes r})^{\Sigma_r}$ is the invariant subring of the $r$-th tensor power of $P$. The symmetric power $\Sym_r(\AA^n)$ can also be viewed as parametrizing the effective $0$-cycles of length $r$ on $\AA^n$. As a map of sets, the Hilbert--Chow morphism $\HC$ sends an ideal $I \subseteq P$ to its vanishing $0$-cycle. One perspective on $\HC$ is therefore that it takes a zero-dimensional polynomial system and outputs its solutions, counted with multiplicities.

The formal definition of $\HC$ as a morphism of schemes is somewhat involved, and many of the standard treatments only define it in a slightly restricted setting, e.g., the treatment in \cite[Section~7.1]{FGA} only yields a map $(\Hilb_r(\AA^n))_{\mathrm{red}} \rightarrow \Sym_r(\AA^n)$ from the underlying reduced scheme of $\Hilb_r(\AA^n)$. The Hilbert--Chow morphism $\HC$ that we use here instead refers to the morphism constructed in \cite[Section~4]{EkedahlSkjelnes2014Hilb} in terms of the \emph{Grothendieck--Deligne norm map}.

By our assumption on the characteristic of $K$, the invariant subring $(P^{\otimes r})^{\Sigma_r}$ is naturally isomorphic to the divided power algebra $\Gamma^r P$ \cite[Section~3.3]{EkedahlSkjelnes2014Hilb}. The latter is formally defined as a certain quotient of the polynomial ring $K[\gamma^r(u) : (r,u) \in \ZZ_{\geq 0} \times P]$ by relations that are modeled on informally thinking of the generator $\gamma^r(u)$ as the expression $\frac{1}{r!} u^r$ but without literally having to divide by $r!$ \cite[Section~1.1]{EkedahlSkjelnes2014Hilb}.

We briefly describe the construction of $\HC$ given in \cite[Section~4.2]{EkedahlSkjelnes2014Hilb}. We will only need the affine case, so let $A$ be a $K$-algebra and define $P_A \coloneqq P \otimes_K A$. According to \emph{loc. cit.}, $\HC(\Spec(A))$ sends an ideal $I \subseteq P_A$ such that $P_A/I$ is a locally free $A$-module of rank $r$ to the $K$-algebra homomorphism
\begin{equation*}
    h_I: \Gamma^{r}P \rightarrow A,
\end{equation*}
which is determined by the following requirement: For any polynomial $u \in P$, the generator $\gamma^{r}(u) \in \Gamma^{r}P$ is mapped by $h_I$ to the determinant of the $A$-linear endomorphism
\begin{equation*}
    \mathrm{mult}_u : P_A/I \rightarrow P_A/I, \ p \mapsto u \cdot p.
\end{equation*}
In formulas,
    \begin{equation*}
        h_I(\gamma^{r}(u)) = \det(\mathrm{mult}_u).
    \end{equation*}
    By \cite[Proposition~1.11]{EkedahlSkjelnes2014Hilb}, the trace $\Tr(\mathrm{mult}_u)$ also lies in the image of $h_I$, namely
    \begin{equation}\label{eqn: trace equation}
        \Tr(\mathrm{mult}_u) = h_I(\gamma^1(u) \ast \gamma^{r-1}(1)).
    \end{equation}
    Here, $\ast$ denotes the external multiplication $\Gamma^i P \times \Gamma^{j} P \rightarrow \Gamma^{i+j} P$ for divided power algebras which, in our setting, agrees with the \emph{shuffle product} in the tensor algebra $P^{\otimes i} \times P^{\otimes j} \rightarrow P^{\otimes (i+j)}$ when restricted to the subring of symmetric tensors, see \cite[Section~3.3]{EkedahlSkjelnes2014Hilb}. In fact, \cite[Proposition~1.11]{EkedahlSkjelnes2014Hilb} shows that every coefficient of the characteristic polynomial of $\mult_u$ is in the image of $h_I$.

\section{Proof of Main Result}\label{SEC:MainResult}

\subsection{Perturbations of elementary symmetric polynomials}
The spectral invariants of Definition~\ref{defn:spectralInvariants} are ``perturbations'' of the elementary symmetric polynomials by lower degree polynomials. After some initial experiments, one might suspect that every ideal generated by such perturbed elementary symmetric polynomials has nonzero solutions unless the ideal is actually equal to $(e_1, \ldots, e_n)$ itself. The following example shows that this is not true even for $n = 3$.

\begin{example}\label{ex:counterexamplePerturbation} 
The ideals $I_1 = (e_1, e_2 - v_1 - v_2, e_3 + v_1v_2-v_1-v_2)$, $I_2 = (e_1,e_2 + v_1, e_3 + v_1^2)$, and $I_3 = (e_1, e_2 - v_2, e_3+v_1^2+v_1v_2+v_2)$ all have the property that their vanishing sets consist only of the origin, and none of them agrees with $(e_1,e_2,e_3)$. In fact, using the Hilbert--Chow technique based on Section~~\ref{sec: AG background} and developed in the following subsections, one can show that, apart from $(e_1,e_2,e_3)$ itself, the $\mathbb{G}_{m}$-orbits of  $I_1$, $I_2$ and $I_3$ for the standard scalar multiplication action of $\mathbb{G}_{m}$ on $\AA^3$ give the full set of ideals that only vanish at the origin and are generated by perturbations of elementary symmetric polynomials.
\end{example}

A straightforward consequence of Lemma~\ref{lem:SpectralInvariants} is that the spectral invariants have the elementary symmetric polynomials as leading terms and only involve \emph{square-free} monomials in the $v_i$. This motivates the following more narrow notion of perturbations.

\begin{defn}\label{defn:idealSpectralInvariants}
    A \textit{perturbation} of $e_i$ is a polynomial $g_i\in P$ of degree $<i$ whose monomials are all square-free. We call $e_i+g_i$ a perturbed elementary symmetric polynomial, and the ideal
    \begin{equation*}
        E\coloneqq (e_1+g_1,\dots, e_n+g_n)\subset P
    \end{equation*}
    an \emph{ideal of generalized spectral invariants}.
\end{defn}

We will ultimately specialize the $g_i$ to the spectral invariants of $A$, but for now we treat their coefficients as independent parameters. We seek to classify which ideals of generalized spectral invariants $E$ satisfy $\mathcal{V}(E) = \{0\}$.

For a subset $J \subseteq [n]$ with $|J| \leq i-1$, denote by $a^{(i)}_{J}$ the coefficient of the monomial
\begin{equation*}
    v^J = \prod_{j \in J} v_j
\end{equation*} 
in $g_i$. For a fixed $i$ and $0\leq k \leq i-1$, we write $a^{(i)}_k$ for the set of coefficients $a^{(i)}_J$ with $|J| = k$. When viewing $a^{(i)}_J$ as variables, we assign them the (\emph{auxiliary}) degree
\begin{equation*}
    \deg(a^{(i)}_J) \coloneqq i - |J| > 0.
\end{equation*}
It is convenient to picture the data of the ideal $E$ in a triangular array (Table \ref{table:1}), where column $i$ corresponds to $e_i+g_i$, the $k$-th row from the bottom collects all degree $k$ terms in the variables $v_1,\dots, v_n$, and the antidiagonals correspond to constant values of the auxiliary degree $i-|J|$.

\begin{table}[h]
\centering
\renewcommand{\arraystretch}{1.2}
\begin{tabular}{@
    {}>{\centering\arraybackslash}m{0pt}@{}|
    >{\centering\arraybackslash}m{0.7cm}|
    >{\centering\arraybackslash}m{0.7cm}|
    >{\centering\arraybackslash}m{0.7cm}|
    >{\centering\arraybackslash}m{0.7cm}|
    >{\centering\arraybackslash}m{0.7cm}|}
\hline
\rule{0pt}{0.6cm} & &  &  &  & $e_n$ \\ \hline
\rule{0pt}{0.6cm} & &  &  & $\Ddots$ & $\vdots$ \\ \hline
\rule{0pt}{0.6cm} & &  & $e_3$ & $\cdots$ & $a^{(n)}_3$ \\ \hline
\rule{0pt}{0.6cm} & & $e_2$ & $a^{(3)}_2$ & $\cdots$ & $a^{(n)}_2$ \\ \hline
\rule{0pt}{0.6cm} & $e_1$ & $a^{(2)}_1$ & $a^{(3)}_1$ & $\cdots$ & $a^{(n)}_1$ \\ \hline
\rule{0pt}{0.6cm} & $a^{(1)}_0$ & $a^{(2)}_0$ & $a^{(3)}_0$ & $\cdots$ & $a^{(n)}_0$ \\ \hline
\end{tabular}
\caption{The ideal of generalized spectral invariants}\label{table:1}
\end{table}

We next place these ideals among a well-behaved one-parameter family. Let $T$ be a new variable, and let $E^H \subset P[T]$ be the homogenization of $E$ with respect to the standard grading in which $\deg T = 1$ and the $v_i$ retain degree 1. Concretely, if $g_i = \sum a_J^{(i)} v^J$, then the homogenized generator of degree $i$ is
\begin{equation*}
    e_i + g_i^H, \qquad g_i^H = \sum_{|J| \leq i-1} a^{(i)}_J v^J T^{i-|J|}.
\end{equation*}

\begin{lemma}\label{lemma:flatFamily}
    The ideal $E^H$ defines a finite flat family of subschemes of $\AA^n$ of length $n!$ or, equivalently, a morphism
    \begin{equation*}
        c: \AA^1 = \Spec K[T] \longrightarrow \Hilb_{n!}(\AA^n).
    \end{equation*}
    If $E$ vanishes only at the origin, then the composition $\HC\circ c$ with the Hilbert-Chow morphism is constant.
\end{lemma}
\begin{proof}
    The  natural map $K[T] \to P[T] / E^H$ makes $P[T]/E^H$ a finitely generated free $K[T]$-module. Since the top-degree parts of the generators of $E$ are $e_1, \ldots, e_n$, which form a regular sequence, so do the generators of $E$ themselves, and we have
    \begin{equation*}
        \dim_K(P/E) = \dim_K(P/(e_1,\dots, e_n)) = n!.
    \end{equation*}
    It follows that $P[T]/E^H$ is a free $K[T]$-module of rank $n!$, and hence $E^H$ defines a flat family of zero-dimensional subschemes of $\AA^n$ of length $n!$. This yields a morphism
    \begin{equation*}
		c: \AA^1 \rightarrow \Hilb_{n!}(\AA^n), \qquad t \mapsto [E^H|_{T = t}].
	\end{equation*}
    On $\AA^1\setminus \{0\}$, the family $E^H|_{T=t}$ is obtained from $E$ by the scaling action of $\mathbb{G}_m$ on $\AA^n$, so the image of $c$ is contained in the closure of a $\mathbb{G}_m$-orbit. The Hilbert--Chow morphism
    \begin{equation*}
        \HC:\Hilb_{n!}(\AA^n)\longrightarrow \Sym_{n!}(\AA^n) \coloneqq (\AA^n)^{n!}/\Sigma_{n!}
    \end{equation*}
    is $\mathbb{G}_m$-equivariant with respect to the induced scaling action. If $E$ vanishes only at the origin, then so does $E^H|_{T=t}$ for every $t\neq 0$, and thus the corresponding $0$-cycle is always $n!\cdot [0]$. Hence $\HC \circ c$ is constant.
\end{proof}

If $\HC\circ c$ is constant, then in particular its differential (and more generally all coefficients of positive powers in the Taylor expansion) of $\HC \circ c$ at $T=0$ vanish. Via the description of $\HC$ in terms of the Grothendieck--Deligne norm map \cite{EkedahlSkjelnes2014Hilb}, this yields explicit relations among the coefficients $a_J^{(i)}$ of the perturbations $g_i$: For example, for any polynomial $u \in P$, all coefficients of positive powers of $T$ in $\Tr(\mult_u)$ must vanish.

We now make this precise and extract certain leading relations in the $a_J^{(i)}$.

\subsection{Polynomiality and leading terms}
Let $E^H$ be as above, and consider the $K[T]$-linear endomorphism
\[
   \mult_u : P[T]/E^H \longrightarrow P[T]/E^H, \qquad p \longmapsto u \cdot p
\]
for some homogeneous $u\in P$ of degree $d$.

\begin{lemma}\label{lemma:polynomiality}
    Let $u \in P$ be homogeneous of degree $d$, and let $r\ge 1$. Then the $r$-th coefficient of the characteristic polynomial of $\mult_u$ is of the form
    \[
        T^{rd}\, R^{u,r},
    \]
    where $R^{u,r}$ is a homogeneous polynomial of degree $rd$ in the variables $a_J^{(i)}$. In particular, for $r=1$ we have
    \[
        \Tr(\mult_u) = T^d R^{u,1},
    \]
    and we abbreviate $R^u\coloneqq R^{u,1}$.
\end{lemma}

\begin{proof}
    Let $E_0\coloneqq E^H|_{T=0} \subseteq P$ and choose a homogeneous $K$-basis $(b_k)$ of $P/E_0$. By graded Nakayama, the same $(b_k)$ is a $K[T]$-basis of the free module $P[T]/E^H$. Let $M=(M_{kl})$ be the matrix of $\mult_u$ in this basis. Then the $r$-th characteristic coefficient is, up to sign, the sum of all principal $r \times r$ minors of $M$, so it suffices to show that each entry $M_{kl}$ is a homogeneous polynomial of degree $\deg(b_l) - \deg(b_k) + d$ in the $a^{(i)}_J$, multiplied by $T^{\deg(b_l) - \deg(b_k) + d}$.

    Since $e_1, \ldots, e_n$ form a regular sequence, we have $E^H = (e_1+g_1^H, \dots, e_n+g_n^H)$, i.e., $E^H$ is generated by the homogenizations of the generators of $E$; in other words, the generators of $E$ form a Gröbner basis for the weight order given by total degree. The homogenized perturbations $g_i^H$ have the property that in each of their terms the exponent of $T$ is exactly the auxiliary degree of the corresponding $a^{(i)}_J$.
    
    Let now $p \in P$ be any homogeneous polynomial. Writing $\overline{p} \in P[T]/E^H$ as a linear combination of the $b_k$, we claim that the coefficient of every $b_k$ is a homogeneous polynomial in the $a^{(i)}_J$ of degree $\deg(p) - \deg(b_k)$, times $T^{\deg(p) - \deg(b_k)}$. Applying this to $p \coloneqq u b_l$ will then finish the proof. To prove the claim, we use induction on $\deg(p)$. The case $\deg(p) = 0$ is clear. Next, since $(b_k)$ is also a basis of $P/E_0$, we may uniquely write $p$ as
    \begin{equation*}
       p = \sum_{j \colon \deg(b_j) = \deg(p)} \alpha_j b_j + b,
    \end{equation*}
    where $\alpha_j \in K$ and $b \in (E_0)_{\deg(p)} = (e_1, \ldots, e_n)_{\deg(p)}$ are in particular independent of the $a^{(i)}_J$. Write $b = \sum_{i = 1}^n p_i e_i$ for some homogeneous $p_i \in P$ of degree $\deg(p) - i$. Modulo $E^H$, we then have
    \begin{equation*}
        b \equiv - \sum_{i=1}^n p_i g_i^H = - \sum_{i=1}^n \sum_{\substack{J_i \subseteq [n] \\ |J_i| < i}} T^{i-|J_i|} a^{(i)}_{J_i} (p_i v^{J_i}).
    \end{equation*}
    Each $p_i v^{J_i} \in P$ is homogeneous of degree $< \deg(p)$, so we may use the induction hypothesis to conclude.
\end{proof} 

We consider next the operation of extracting a certain ``top part'' of $R^u$, namely the part that is \emph{linear} when assigning degree $1$ to all the $a^{(i)}_J$ instead of the auxiliary degree.

\begin{defn}
    Let $u\in P$ be homogeneous of degree $d$, and let $R^u$ be as in Lemma \ref{lemma:polynomiality}. We define $R^u_{\mathrm{top}}$ to be the sum of only those terms of $R^u$ whose variables lie in antidiagonals of total auxiliary degree $d$ in Table~\ref{table:1}, i.e., we set to zero every variable $a^{(i)}_J$ with $i-|J| < d$.
\end{defn}
Equivalently, $R^u_{\mathrm{top}}$ is obtained by the following universal specialization.

\begin{lemma}\label{lemma:top}
        Let $u$ and $d$ be as above. Let
    \[
        \phi: K[(a^{(i)}_J)_{i,J}] \longrightarrow K[(a^{(i)}_J)_{i,J}]
    \]
    be the $K$-algebra endomorphism which, for each cell in Table~\ref{table:1} of auxiliary degree $< d$, sends all variables in that cell to the same scalar in $K$, sends $a^{(i)}_\emptyset$ to $0$ for all $i<d$, and fixes all remaining variables. Then
    \[
        \phi(R^u) = R^u_{\mathrm{top}}.
    \]
\end{lemma}

\begin{proof}
    Since the characteristic polynomial of $\mult_u$ only depends on the ideal $E^H$, the polynomial $R^u$ is unchanged if we replace the generators $e_i+g_i^H$ by any other set obtained from them by adding scalar multiples of earlier generators $e_j+g_j^H$ with $j<i$ (a ``triangular change of generators'' in Table~\ref{table:1}). In particular, after applying~$\phi$ we may use such changes to eliminate all terms in $g_i^H$ of degree $>i-d$ in $v_1,\dots,v_n$ without affecting the terms of degree exactly $i-d$, in a manner resembling Gaussian elimination. For the resulting generators, the trace $\Tr(\mult_u)$ is precisely $T^d R^u_{\mathrm{top}}$. Since we have not altered the ideal, the equality $\phi(R^u)=R^u_{\mathrm{top}}$ follows.
\end{proof}

The final two remarks record the particularly simple form of the top part $R^u_{\mathrm{top}}$ and the natural symmetry of $R^{u,r}$ under permutations and column operations.

\begin{rmk}\label{rmk:single_perturbation}
    In the situation of Lemma~\ref{lemma:top}, $R^u_{\mathrm{top}}$ is homogeneous and \emph{linear} in the variables $a^{(i)}_J$ with $i-|J| = d$. In particular, for each $0\le k\le n-d$, consider the specialization which sets to zero all variables in the $d$-th antidiagonal of Table~\ref{table:1} except those in the cell $a^{(k+d)}_k$; denote the resulting polynomial by $R^{u}_{\mathrm{top},k}$. Then
    \[
        R^{u}_{\mathrm{top}} = \sum_{k=0}^{n-d} R^{u}_{\mathrm{top},k}.
    \]
    Moreover, $R^{u}_{\mathrm{top},k}$ is the trace of $\mult_u$ for all ideals of generalized spectral invariants of the form
    \[
       (e_1,\dots,e_{k+d-1},\, e_{k+d}+g_{k+d},\, e_{k+d+1},\dots,e_n),
    \]
    where $g_{k+d}$ is homogeneous of degree $k$ in $v_1,\dots,v_n$. $\hfill \diamond$
\end{rmk}

\begin{rmk}\label{rmk:symmetry}
    For any $\sigma \in \Sigma_n$ acting by
    \[
       \sigma(a_J^{(i)}) \coloneqq a_{\sigma^{-1}(J)}^{(i)},\qquad
       \sigma(v_j)\coloneqq v_{\sigma(j)},
    \]
    we have
    \[
       R^{\sigma(u),r} = \sigma^{-1}\bigl(R^{u,r}\bigr).
    \]
    In particular, if $\sigma(u)=u$, then $\sigma$ fixes $R^{u,r}$. In the special case $r=1$, the map $u\mapsto R^u$ is $K$-linear because so is $u\mapsto \Tr(\mult_u)$.$\hfill \diamond$
\end{rmk}

\subsection{Explicit leading relations} We now fix a choice of $u$ and compute $R^u_{\mathrm{top}}$ explicitly.

\begin{proposition}\label{prop:relations}
    Let $1 \leq m \leq n$ and set $u\coloneqq v_1 \cdots v_m \in P$. Let $R^u_{\mathrm{top}} = \sum_{k=0}^{n-m} R^u_{\mathrm{top},k}$ be as in Remark \ref{rmk:single_perturbation}. Then, for all $0\leq k \leq n-m$,
    \begin{equation*}
        R^u_{\mathrm{top},k}
        = m \sum_{j=0}^{\min(k,m)} (-1)^{j+1}\,(m+k-j-1)!\,(n-m-k+j)!\!
        \sum_{\substack{J\subseteq [n]\\ |J|=k\\ |J\cap [m]|=j}} a_J^{(m+k)}.
    \end{equation*}
\end{proposition}
\begin{proof}
 By Remarks~\ref{rmk:symmetry} and~\ref{rmk:single_perturbation}, since $\sigma(u) = u$ for all $\sigma \in \Sigma_{\{1, \ldots, m\}} \times \Sigma_{\{m+1,\ldots,q\}} \subseteq \Sigma_q$,  there exists scalars $\lambda_{k,j} \in K$ such that 
 \begin{equation*}
     R^u_{\mathrm{top},k}
        = \sum_{j=0}^{\min(k,m)} \lambda_{k,j}
          \sum_{\substack{J\subseteq [n]\\ |J|=k\\ |J\cap[m]|=j}} a_J^{(m+k)}.
 \end{equation*}
 Moreover, for $k\geq 1$, $R^u_{\mathrm{top},k}$ is unchanged if we add a multiple of $e_k$ to $e_{k+m}+g_{k+m}$, hence
    \begin{equation}\label{eq:lambda_kj}
        \sum_{j=0}^{\min(k,m)} \lambda_{k,j} \binom{m}{j}\binom{n-m}{k-j} = 0 \quad\text{for all }k\ge 1.
    \end{equation}
    We claim that we also have the recursion
    \begin{equation}\label{eq:recursion}
        \lambda_{k,j} = -\,\lambda_{k-1,j-1}\qquad\text{for all }k\ge j\ge 1.
    \end{equation}
    Assuming \eqref{eq:recursion}, the value of $\lambda_{0,0}$ together with \eqref{eq:lambda_kj} determines all $\lambda_{k,j}$ uniquely.

    To compute $\lambda_{0,0}$, note that
    \[
        R^u_{\mathrm{top},0} = \lambda_{0,0}\,a^{(m)}_\emptyset.
    \]
    Setting $a^{(m)}_\emptyset=1$, $\lambda_{0,0}$ is precisely the trace of $\mult_u$ for the ideal
    \[
       (e_1,\dots,e_{m-1},\,e_m+1,\,e_{m+1},\dots,e_n),
    \]
    which is $\Sigma_n$-invariant. Thus
    \begin{align*}
        \lambda_{0,0}
        &= \Tr(\mult_u)
         = \frac{1}{\binom{n}{m}}\sum_{\overline{\sigma}\in \Sigma_n/(\Sigma_m\times \Sigma_{n-m})} \Tr\bigl(\mult_{\sigma(u)}\bigr)
         = \frac{\Tr(\mult_{e_m})}{\binom{n}{m}} \\
        &= \frac{\Tr(\mult_{-1})}{\binom{n}{m}}
         = \frac{-n!}{\binom{n}{m}}
         = -\,m!\,(n-m)!.
    \end{align*}
    This matches the coefficient in the formula of the proposition for $k=0$.

    It is straightforward to verify that the coefficients in the statement of the proposition satisfy the recursion~\eqref{eq:recursion}. We now check they satisfy \eqref{eq:lambda_kj}. For $k\ge 1$ we have
    \begin{align*}
        &\ \ \ \ \, \sum_{j=0}^{\min(k,m)} \lambda_{k,j} \binom{m}{j}\binom{n-m}{k-j} \\
        &= m \sum_{j=0}^{k} (-1)^{j+1} (m+k-j-1)!\,(n-m-k+j)!\,\binom{m}{j}\binom{n-m}{k-j} \\
        &= -\,\frac{m}{k} m!\,(n-m)!\sum_{j=0}^k (-1)^j \binom{k}{j}\binom{m+k-j-1}{m-j}.
    \end{align*}
    The last sum is the coefficient of $t^m$ in
    \[
       \Bigl(\sum_{j=0}^k (-1)^j\binom{k}{j}t^j\Bigr)
       \Bigl(\sum_{j'\ge 0}\binom{j'+k-1}{j'} t^{j'}\Bigr)
       = (1-t)^k \cdot (1-t)^{-k} = 1,
    \]
    hence vanishes. This proves \eqref{eq:lambda_kj}.

    It remains to prove recursion~\eqref{eq:recursion}. For this, let $M$ be the matrix representing $\mult_u$ in the Artin monomial basis. Our argument will in fact show that recursion~\eqref{eq:recursion} even holds on the level of the individual diagonal entries $M_{xx}$.

    Let $x = v_1^{l_1} \cdots v_{n-1}^{l_{n-1}}$ be an Artin monomial. Since $u = v_1 \cdots v_m$ has degree $m$, the entry $M_{xx}$ is a homogeneous polynomial of degree $m$ in the $a^{(m+k)}_J$, $|J| = k$, times $T^m$, by the proof of Lemma~\ref{lemma:polynomiality}. In particular, we may base change to the Artinian $K$-algebra $A \coloneqq K[T]/(T^{m+1})$ and regard $E^H$ as an ideal in $P_A = P \otimes_K A$.

    We abbreviate $v \coloneqq v_{m+1}^{l_{m+1}} \cdots v_{n-1}^{l_{n-1}}$. We have $ux = v_1^{l_1 + 1} y v$ with $y \coloneqq v_2^{l_2 + 1} \cdots v_m^{l_m + 1}$. Let $y' = \sum_j \alpha_j b_j$ be the Artin reduction of $y$ with $\alpha_j \in K$ and Artin monomials $b_j$. If $\alpha_j \neq 0$, then $b_j$ does not contain the variables $v_{m+1}, \ldots, v_n$ by the proof of Lemma~\ref{lemma:ArtinMonomialReduction}. Let $p \coloneqq y - y' \in (e_1, \ldots, e_n)$ and write $p = \sum_{i=1}^n p_i e_i$ with $p_i$ homogeneous of degree $\deg(p) - i$. We have
    \begin{equation}\label{eq:ux}
        ux = v_1^{l_1+1} y' v + v_1^{l_1 + 1} p v.
    \end{equation}
    For the second summand in \eqref{eq:ux}, we observe that modulo $E^H$ we have
    \begin{equation*}
        v_1^{l_1 + 1} p v = \sum_{i=1}^n v_1^{l_1+1} p_i e_i v \equiv - T^m v_1^{l_1 + 1} p_{m+k} g_{m+k} v.
    \end{equation*}
    Since $T^{m+1} = 0$ in $P_A$, the $x$-coefficient of $v_1^{l_1 + 1} p v$ as an element in the quotient $P_A/E^H$ agrees with the Artin reduction $\xi$ of the polynomial $-v_1^{l_1 + 1} p_{m+k} g_{m+k} v \in P$, up to multiplication by $T^m$. But by Lemma~\ref{lemma:ArtinMonomialReduction}, every Artin monomial occurring in $\xi$ has $v_1$-exponent $\geq l_1 + 1$, in particular $x$ itself does not occur. Therefore, the second summand of \eqref{eq:ux} does not contribute to the diagonal entry $M_{xx}$.
    Turning to the first summand of \eqref{eq:ux}, write $b_j = v_1^{s_j} \tilde{b}_j$ such that $\tilde{b}_j$ is not divisible by $v_1$. Then
    \begin{equation*}
        v_1^{l_1+1} y' v = \sum_j \alpha_j v_1^{l_1 + 1 + s_j} \tilde{b}_j v.
    \end{equation*}
    We observe that those summands with $l_1 + 1 + s_j \leq n-1$ are Artin monomials of degree $\deg(x) + m > \deg(x)$, so they cannot contribute to $M_{xx}$. Continuing with the remaining summands for which $l_1 + 1 + s_j \geq n$, we use Corollary~\ref{cor:v1^q} to write
    \begin{align*}
        \sum_{j \colon l_1 + 1 + s_j \geq n} \alpha_j v_1^{l_1 + 1 + s_j} \tilde{b}_j v &= \sum_{j \colon l_1 + 1 + s_j \geq n} \alpha_j \tilde{b}_j v \sum_{i = 1}^n (-1)^{i-1}  v_1^{l_1 + 1 + s_j - i} e_i \\
        &\equiv T^m \cdot (-1)^{m+k} \sum_{j \colon l_1 + 1 + s_j \geq n} \alpha_j \tilde{b}_j v v_1^{l_1 + 1 + s_j - m - k} g_{m+k},
    \end{align*}
    modulo $E^H$. We denote the coefficient of $T^m$ in the last expression by $W_k(g_{m+k})$, viewed as a polynomial in $P$ for every choice of perturbation $g_{m+k}$ (equivalently, for every choice of $a^{(m+k)}_J \in K$, $|J| = k$). Since $T^{m+1} = 0$ in $P_A$, we conclude once more that the coefficient of $x$ of the first summand in $\eqref{eq:ux}$ agrees with the coefficient of $x$ in the Artin reduction of $W_k(g_{m+k})$. The final observation is that, if $h$ is a square-free monomial in $v_2, \ldots, v_n$ of degree $k-1$, then, as polynomials in $P$, we have $W_k(v_1h) = - W_{k-1}(h)$. Since taking the Artin reduction is $K$-linear, the coefficient of $x$ in the Artin reduction of $W_k(v_1h)$ is the negative of the coefficient of $x$ in the Artin reduction of $W_{k-1}(h)$. But now, setting $g_{m+k}$ equal to $v_1 h$, the trace of $\mult_u$ is exactly $\lambda_{k,j}$, where $j$ is the number of variables in $v_1 h$ with indices in $\{1,2,\ldots,m\}$. Similarly, for $k-1$ instead of $k$, setting $g_{m+k-1}$ equal to $h$, the trace of $\mult_u$ is exactly $\lambda_{k-1,j-1}$. This completes the proof of recursion~\eqref{eq:recursion}.
\end{proof}

With Proposition~\ref{prop:relations}, as we will see, we have established just enough relations to prove our first main theorem.

\subsection{\emph{Proof of Theorem~\ref{thm:IntroMain}}} 
One direction is straightforward. If, for each $1\le k \le n$, all principal $k \times k$ minors of $A$ are equal, then by Lemma~\ref{lem:SpectralInvariants} each spectral invariant $S_i$ is a linear combination of elementary symmetric polynomials of degrees $\le i$ with top-degree part equal to $e_i$ and without constant terms. In particular,
    \[
       (S_1,\dots,S_n) = (e_1,\dots,e_n),
    \]
    so the only common zero is the origin. Thus there is no nonzero diagonal matrix $D$ with $A+D$ isospectral to $A$.

 For the converse, let $E\coloneqq (S_1,\dots,S_n)$ be the ideal of spectral invariants of $A$. Assume that $E$ vanishes only at the origin. We must show that for every $1 \le m \le n$ all principal $m \times m$ minors of $A$ agree.

  We proceed by induction on $m$. For $m=1$, the statement is equivalent to the equality of all diagonal entries of $A$, which follows immediately from the vanishing of $R^{v_1}$ (see below). Suppose now that for some $0 \le d \le n-1$ all principal $i \times i$ minors of $A$ agree for each $i \le d$. We show that all principal $(d+1) \times (d+1)$ minors are equal.

  By Lemma~\ref{lem:SpectralInvariants}, the assumption for $i\le d$ implies that in the table description of $E$ (Table~\ref{table:1}), all variables in each cell of auxiliary degree $\le d$ are equal. By Lemma~\ref{lemma:top}, for any homogeneous $u$ of degree $d+1$ the evaluation of $R^u$ at $E$ agrees with the evaluation of $R^u_{\mathrm{top}}$ at $E$. In particular, this holds for $u=v_1\cdots v_{d+1}$.

   On the other hand, by Lemma~\ref{lemma:flatFamily}, under the hypothesis that $E$ vanishes only at the origin, the composition $\HC\circ c$ is constant, hence for any homogeneous $u$ of positive degree we must have $R^u(E)=0$. For $u=v_1\cdots v_{d+1}$ this yields
    \[
        R^{v_1\cdots v_{d+1}}_{\mathrm{top}}(E) = R^{v_1\cdots v_{d+1}}(E) = 0.
    \]
    We claim that this forces equality of all principal $(d+1) \times (d+1)$ minors of $A$, completing the inductive proof.

    Set $m\coloneqq d+1$. Using Proposition~\ref{prop:relations} and Remark~\ref{rmk:single_perturbation}, we can write
    \[
       R^{v_1\cdots v_m}_{\mathrm{top}} = \sum_{k=0}^{n-m} R^{v_1\cdots v_m}_{\mathrm{top},k}.
    \]
    Evaluating at $E$ and using Lemma~\ref{lem:SpectralInvariants} to substitute
    \[
       a_J^{(m+k)} = \sum_{\substack{N\subseteq [n]\setminus J\\ |N|=m}} \det(A_N),
    \]
    we get
    \begin{align*}
       R^{v_1\cdots v_m}_{\mathrm{top}}(E)
         &= \sum_{k=1}^{n-m} \sum_{j=0}^{\min(k,m)} \lambda_{k,j}
            \sum_{\substack{J\subseteq [n]\\ |J|=k\\ |J\cap [m]|=j}}
            \sum_{\substack{N\subseteq [n]\setminus J\\ |N|=m}} \det(A_N) \\
         &= \sum_{\substack{N\subseteq [n]\\ |N|=m}} \det(A_N)
            \sum_{\emptyset\neq J\subseteq [n]\setminus N} \lambda_{|J|,|J\cap [m]|},
    \end{align*}
    with $\lambda_{k,j}$ as in the proof of Proposition~\ref{prop:relations}. (The $k=0$ term vanishes because the $S_i$ have no constant terms.) We can rewrite
    \[
       \sum_{\emptyset\neq J\subseteq [n]\setminus N} \lambda_{|J|,|J\cap [m]|}
       = m!\,(n-m)! + \sum_{J\subseteq [n]\setminus N} \lambda_{|J|,|J\cap [m]|}.
    \]
    We set $r \coloneqq |[m] \setminus N|$ and decompose $J$ as $J = J_1 \sqcup J_2$ such that $J_1 \subseteq [m] \setminus N$ and $J_2 \subseteq [n] \setminus ([m] \cup N)$ with $|J_1| = j$ and $|J_2| = j'$. This yields
    \begin{equation*}
         \sum_{J \subseteq [n] \setminus N} \lambda_{|J|,|J \cap [m]|} = - m \sum_{j = 0}^r (-1)^j \binom{r}{j} \sum_{j' = 0}^{n-m-r} \binom{n-m-r}{j'} (m-1+j')! (n-m-j')!.
    \end{equation*}
   The sums over $j$ and $j'$ are independent, and $\sum_{j = 0}^r (-1)^j \binom{r}{j} = 0$ if $r \geq 1$, so we are left only with the term for $r = 0$, i.e., $N = [m]$. In this case, we get
    \begin{align*}
         \sum_{J \subseteq [n] \setminus [m]} \lambda_{|J|,|J \cap [m]|} &= -m \sum_{j' = 0}^{n-m} \binom{n-m}{j'} (m-1+j')! (n-m-j')! \\
         &= -m \sum_{j' = 0}^{n-m} \binom{m-1+j'}{m-1} (m-1)! (n-m)! \\
         &= - m! (n-m)! \sum_{j' = 0}^{n-m} \binom{m-1+j'}{m-1} \\
         &= - m! (n-m)! \binom{n}{m} = - n!.
    \end{align*}
    Putting everything together, we obtain
    \[
       R^{v_1\cdots v_m}_{\mathrm{top}}(E)
       = m!(n-m)! \sum_{\substack{N\subseteq [n]\\ |N|=m}} \det(A_N)
         \;-\; n!\,\det(A_{[m]}).
    \]
    Since $\operatorname{char}K=0$ or $>n$, the integer $n!$ is invertible in $K$, so this expression vanishes if and only if
    \begin{equation}\label{eq:principalMinors}
       \det(A_{[m]})
       = \binom{n}{m}^{-1}\sum_{\substack{N\subseteq [n]\\ |N|=m}} \det(A_N),
    \end{equation}
    i.e., the principal $m \times m$ minor on indices $[m]$ equals the average of all principal $m \times m$ minors.

    Finally, conjugating $A$ by any permutation matrix corresponding to $\sigma\in \Sigma_n$ leaves the right-hand side of \eqref{eq:principalMinors} invariant, while the left-hand side becomes $\det(A_{\sigma([m])})$. Since every principal $m \times m$ minor arises in this way, we conclude that all such minors are equal. \qed

 \begin{proof}[Proof of Corollary~\ref{cor: adjacency matrix}]
     Follows immediately as all $2 \times 2$ principal minors must be equal.
 \end{proof}
\section{Floquet Isospectrality}~\label{SEC:Floquet Isospectrality}

Periodic operators serve as ``tight-binding" models whose spectra describe electron dynamics in crystalline solids such as those found in many nano-materials and topological insulators \cites{ksurvey, Kittel}. The study of Floquet isospectrality of such operators has a rich history that began around four decades ago ~\cites{kapiii,kapi,kapii,liujde, flmrp, saburova, VPapan, liu2021fermi, Burak}. One continuing thread of research has been finding potentials that are isospectral to the zero potential $\bfz$ ~\cites{ERT84,MT76,ERTII,gki,wa,gui90,eskin89}.  

The spectra of discrete periodic operators can be realized as the projections of algebraic varieties called Bloch varieties, and there have been many recent papers dedicated to analyzing spectral problems through algebraic methods, see \cites{kuchment2023analytic, shipman2025algebraic, liujmp22} are references therein.

\subsection{Operators on $\ZZ^d$} \label{sec: Operators on Periodic Graphs}
Let $\Delta$ denote the discrete Laplacian acting on square-summable functions $u \in \ell^2(\ZZ^d)$ by
\begin{equation*}
	(\Delta u)(n)=\sum_{|n'-n|_1=1} u(n'),
\end{equation*}
where, as usual, for $n, n' \in \ZZ^d$ we write $|n'-n|_1 := \sum_{i=1}^d |n'_i-n_i|$.
We will use the shorthand $Q\ZZ \coloneqq q_1\ZZ \oplus \dots \oplus q_d \ZZ$ whenever $Q = (q_1, \ldots, q_d) \in \ZZ^d_{>0}$ is a tuple of periods. A potential $V:\ZZ^d \to \CC$ is said to be \emph{$Q\ZZ$-periodic} if
\begin{equation*}
V(n+m)=V(n),\,\,\text{for all}\,\, n\in \mathbb{Z}^d\,\,\text{and}\,\,m\in Q\ZZ.
\end{equation*}

By adding a $Q\ZZ$-periodic potential $V$, acting as a multiplication operator, to our Laplacian, we obtain the discrete periodic Schr\"odinger operator $\Delta + V$. Typically, one is interested in the spectrum of this operator acting on $\ell^2(\ZZ^d)$:
\begin{equation}\label{eq1}
(\Delta u)(n) + V(n)u(n) = \lambda u(n),\,\,\text{for all}\, n \in \ZZ^d.
\end{equation}

Floquet theory reveals that the generalized eigenfunctions of $\Delta + V$ are exactly $Q\ZZ$-quasiperiodic functions with Floquet multiplier in $\TT^d$; that is, functions $u : \ZZ^d \to \CC$ 
subject to the Floquet boundary condition
\begin{equation}~\label{eq2}
u(n + q_j\delta_j) = z_j u(n), \text{ for all } n \in \ZZ^d \,\,,z_j\in \TT\,\, \text{and } j \in [d] = \{1,\ldots, d\},
\end{equation}
where 
$\{\delta_j\}_{j=1}^d$ is the standard basis of $\R^d$ and $\TT$ is the complex unit circle.

Writing $q = \prod_{i=1}^d q_i$, equation~\eqref{eq1} with the boundary condition~\eqref{eq2} can be realized as the eigen-equation of a $q \times q$ matrix $L_V(z)$, where $z = (z_1, \dots, z_d)$.  In particular, when $d = 1$, $\Delta + V$ can be realized as matrix multiplication by the following finite matrix:       \begin{equation}\label{eq:defmatricesdim1}  L_{V} = \begin{cases}
       \begin{pmatrix}
            V(1) & 1 & 0 & \dots & 0 & z_1^{-1} \\
            1 & V(2) & 1 & 0 & \dots & 0 \\
            0 & 1 & \ddots & \ddots & \ddots & \vdots \\
            \vdots & \ddots & \ddots & \ddots & \ddots & 0\\ 
            0 & 0 & \dots & 1 & V(q_1-1) & 1 \\
            z_1 & 0 & \dots & 0 & 1 & V(q_1) \\
        \end{pmatrix} & \text{for }q_1>2, \\ \begin{pmatrix}
            V(1) & 1 +z_1^{-1} & \\
            1 + z_1 & V(2)  \end{pmatrix} & \text{for }q_1=2, \\ 
            \quad V(1) + z_1 + z_1^{-1} & \text{for } q_1=1.  \end{cases}\end{equation}
For $d >1$, $\Delta + V$ has the following recursive matrix structure:
 \begin{equation}\label{eq:defmatricesgen}
        L_{V}=\begin{cases}\begin{pmatrix}
            L_{\tilde{V}_1} & I_{\tilde{q}} & 0 & \dots & 0 & z_d^{-1} I_{\tilde{q}} \\
            I_{\tilde{q}} & L_{\tilde{V}_2} & I_{\tilde{q}} & 0 & \dots & 0 \\
            0 & I_{\tilde{q}} & \ddots & \ddots & \ddots & \vdots \\
            \vdots & \ddots & \ddots & \ddots & \ddots & 0\\ 
            0 & 0 & \dots & I_{\tilde{q}} & L_{\tilde{V}_{q_{d}-1}} & I_{\tilde{q}} \\
            z_d I_{\tilde{q}} & 0 & \dots & 0 & I_{\tilde{q}} & L_{\tilde{V}_{q_d}} \\
        \end{pmatrix} & \text{for }q_d>2, \\ \begin{pmatrix}
            L_{\tilde{V}_1} & I_{\tilde{q}} +z_d^{-1}I_{\tilde{q}} & \\
            I_{\tilde{q}} + I_{\tilde{q}}z_d & L_{\tilde{V}_2}  \end{pmatrix} & \text{for }q_d=2, \\ \quad \ \ 
            L_{\tilde{V}_1} + z_1I_{\tilde{q}} + z_1^{-1}I_{\tilde{q}} &\text{for } q_d=1, \end{cases}  \end{equation}
where $\tilde{q} = \frac{q}{q_d}$, and $\tilde{V}_i$ is a $q_1\ZZ \times \cdots \times q_{d-1}\ZZ$-periodic potential given by $\tilde{V}_i(n_1,\dots, n_{d-1}) = V(n_1,\dots, n_{d-1},i)$.

Denote by $\sigma_V(z):=\sigma(L_V(z))$ the (multi)set of eigenvalues of $L_V(z)$, including algebraic multiplicity.  It is a well-known consequence of Floquet theory that
\begin{equation*}
    \sigma(\Delta + V) = \bigcup_{z \in \TT^d} \sigma_V(z).
\end{equation*}

\begin{definition}\label{def:FIso}
Assume that $V$ and $V'$ are $Q\ZZ$-periodic potentials. Two operators $\Delta + V$ and $\Delta + V'$ are called \emph{Floquet isospectral} if $\sigma_V(z) = \sigma_{V'}(z)$ for all $z \in \TT^d$. In this case, it is also common to say that the $Q\ZZ$-periodic potentials $V$ and $V'$ are Floquet isospectral.
\end{definition}

Treating $z$ as indeterminates, we may view $L_{V}(z)$ as a matrix with Laurent polynomial entries in $\CC[z_1^{\pm},\dots, z_d^{\pm}]$. It is easy to see that $L_{V}(z) = L_V^T(z^{-1})$. Denoting the characteristic polynomial of $L_{V}(z)$ by $D_V(z,\lambda)$, we call the vanishing set of $D_V(z,\lambda)$, restricted to $z \in \TT^d$, the \emph{dispersion relation}, and the variety of $D_V(z,\lambda)$ in $(\CC^\ast)^d \times \CC$ the \emph{Bloch variety}. It is easy to see that two potentials $V$ and $V'$ are Floquet isospectral if and only if $D_V(z,\lambda)$ and $D_{V'}(z,\lambda)$ are the same polynomial. 

For a polynomial $f = \sum_{(a,b)} c_{a,b} z^a\lambda^b \in \CC[z_1^{\pm},\dots, z_d^{\pm},\lambda]$ we define $[z^a\lambda^b] f := c_{a,b}$ to denote the coefficient of the monomial $z^a \lambda^b$ in $f$. Clearly, two polynomials $f$ and $g$ are equal if and only if $[z^a\lambda^b] f = [z^a\lambda^b] g$ for all $(a,b) \in \ZZ^{d+1}$.

By Definition~\ref{def:FIso}, it follows that two potentials $V$ and $V'$ are Floquet isospectral if and only if the following system of equations is satisfied:
\begin{equation}\label{eqn:spectral_invariants}
[z^a\lambda^b] (D_V(z,\lambda) - D_{V'}(z,\lambda)) = 0, \text{ for all } (a,b).
\end{equation}

Treating the values taken by $V$ and $V'$ on $\ZZ^d / Q\ZZ$ as indeterminates $v_i$ and $v'_i$, \eqref{eqn:spectral_invariants} becomes a system of polynomials. These polynomials are called the spectral invariants, and together they generate the ideal of spectral invariants; compare with Definition~\ref{defn:idealSpectralInvariants}.

\begin{example}
Consider the case where $Q = (3,2)$. The Floquet matrix $L_V$ of $\Delta+V$ for the square lattice with periodicity depicted in Figure~\ref{fig:squarelat} is \[ L_V = \begin{pmatrix}
    V(1,1) & 1 & z_1^{-1} & 1+ z_2^{-1} & 0 & 0 \\
    1 & V(2,1) & 1 & 0 & 1+ z_2^{-1} & 0 \\ 
    z_1 & 1 & V(3,1) & 0 & 0 & 1+z_2^{-1}\\
     1+ z_2 & 0 & 0 & V(1,2) & 1 & z_1^{-1} \\
    0 & 1+ z_2 & 0 & 1 & V(2,2) & 1  \\ 
     0 & 0 & 1+ z_2 & z_1 & 1 & V(3,2)\\
\end{pmatrix}.\] 
Computing the spectral invariants for $V' = \bfz$, one finds there are more minimal generators than variables. That is, the corresponding ideal of spectral invariants $E_{\Delta+V'} = E_{\Delta}$ is not a complete intersection; however, as $\bfz$ is a trivial solution we know the variety of $E_\Delta$ is nonempty. Symbolic computation reveals that $E_\Delta$ has degree $51$ and the vanishing set contains only $\bfz$. When $V'$ is instead chosen to be generic, the degree of $E_{\Delta+V'}$ is $12$, and the points in the vanishing set of $E_{\Delta+V'}$ arise exactly from the symmetries of the graph depicted in Figure~\ref{fig:squarelat} \cite{kapii}.  \qedhere
\end{example}

\begin{figure}
    \centering
    \includegraphics[width=0.4\linewidth]{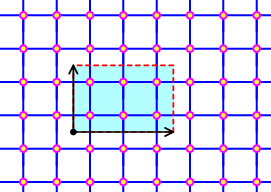}
    \caption{The square lattice with a highlighted fundamental domain that is $(3,2)\ZZ$-periodic.}
    \label{fig:squarelat}
\end{figure}

\subsection{Proof of Theorem~\ref{Cor:FloquetMain}}~\label{Sec:FloqCor}
Let $Q = (q_1,\dots, q_d) \in \NN^d$ and consider a $Q\ZZ$-periodic potential $V$. As sketched in the introduction, Theorem~\ref{Cor:FloquetMain} follows after establishing two cases:
\begin{enumerate} 
        \item \label{case 1} when the periods are small ($q_i \leq 3$ for all $1\leq i \leq d$), and
        \item \label{case 2} when the periods are large ($q_i > 3$ for some $1\leq i \leq d$).
    \end{enumerate}
Case (1) is addressed in Section \ref{sec: small periods} while case (2) is addressed in Section \ref{sec: large periods}. 

When $P\ZZ$ is a rectangular subgroup of $Q\ZZ$, that is $P = (p_1, \dots, p_d) \in \NN^d$ such that $q_i | p_i$ for each $i$, we write $V_P$ for the $P\ZZ$-periodic potential $V_P(n) = V(n)$. Letting $P/Q = (\frac{p_1}{q_1}, \dots, \frac{p_d}{q_d})$, it is well known that $D_{V_P}(z,\lambda)$ and $D_{V}(z,\lambda)$ have the relationship
\begin{equation}~\label{eq:finFourier}
D_{V_P}(z^{P/Q},\lambda) = \prod_{\mu \in U_{P/Q}} D_V(\mu z,\lambda),
\end{equation}
where $U_{P/Q} = U_{p_1/q_1} \times \cdots \times U_{p_d/q_d}$ and $U_{p_i/q_i}$ is the set of $(p_i/q_i)$-th roots of unity. For the reader's convenience, we show how to obtain~\eqref{eq:finFourier}. See~\cite{KuchBook} for more details.

\begin{proof}[Proof of ~\eqref{eq:finFourier}]
Fix $z \in \TT^d$, and consider the space of quasiperiodic functions $\psi$ on $\ZZ^d$ satisfying \begin{equation}\label{eq:newquasi} \psi(n+p_j \delta_j) = z_j^{\frac{p_j}{q_j}} \psi(n) , \text{ for all } n \in \ZZ^d \,\, \text{and } j \in [d].\end{equation}

As such quasiperiodic functions are determined by their values on the fundamental domain $\widehat{P} = \{ (n_1,\dots, n_d) \mid 0 \leq n_i < p_i \} \subset \ZZ^d$, we may represent each quasiperiodic function of ~\eqref{eq:newquasi} as a vector of their values on $\widehat{P}$, thus $\Delta + V_P$ acts on $\psi$ as finite matrix multiplication by $L_{V_P}(z^{P/Q})$, which has characteristic polynomial $D_{V_P}(z^{P/Q},\lambda)$. 

Let $G := Q\ZZ / P\ZZ = \prod_i \ZZ/(p_i/q_i)\ZZ$, and represent elements of $G$ on $\ZZ^d$ by $a \in G \mapsto \widehat{a} = aQ = (a_1 q_1, \dots, a_d  q_d) \in \ZZ^d$. Given any quasiperiodic function $\psi$ satisfying ~\eqref{eq:newquasi}, define \begin{equation*} \demph{\psi_\rho(n)} := \frac{1}{|G|} \sum_{a \in G} \overline{\rho^a}  \psi(n+ \widehat{a}), \text{ for } \rho \in U_{P/Q},\end{equation*} where $\rho^a = \rho_1^{a_1} \cdots \rho_d^{a_d}$.   Finite Fourier inversion yields \begin{equation*} \psi(n) = \sum_{\rho \in U_{P/Q}} \psi_{\rho}(n).\end{equation*}  Moreover, by direct computation, we have that 
\begin{equation*}~\label{eq:neqbasis}  \psi_{\rho}(n+ q_j\delta_j) = \rho_j z_j \psi_{\rho}(n), \text{ for all } n\in\ZZ^d  \text{ and } j \in [d].\end{equation*}
That is, each $\psi_\rho$ is a $Q\ZZ$-quasiperiodic function with Floquet multiplier $\rho z$, and is thus determined by its values on $\widehat{Q} = \{ (n_1,\dots, n_d) \mid 0 \leq n_i < q_i \} \subset \ZZ^d$.  Now let us consider how $\Delta + V_P$ acts on each $\psi_\rho$, in particular, we will show that $(\Delta + V_P) \psi_\rho$ only depends on $\psi_\rho$. We do this by looking at the action of $\Delta$ and $V_P$ on $\psi_\rho$ separately, via direct computation. For $\Delta$, we have \begin{equation*} \begin{aligned}
    & \Delta \psi_\rho(n)  =\sum_{|n-n'|_1=1} \psi_\rho(n').
\end{aligned} 
\end{equation*}   Noting that $V_P$ is $Q\ZZ$-periodic, we have that $V_P(n+\widehat{a}) =V_P(n)$ for all $a \in G$ and so \begin{equation*}~\label{eq:potentialrhoact} \begin{aligned}
    &V_P \psi_\rho(n) = V_P\frac{1}{|G|} \sum_{a \in G} \overline{\rho^a}  \psi(n+\widehat{a}) =  \frac{1}{|G|} \sum_{a \in G} \overline{\rho^a} V_P(n+\widehat{a}) \psi(n+\widehat{a})\\
    &= \frac{1}{|G|} \sum_{a \in G} \overline{\rho^a} V_P(n) \psi(n+\widehat{a}) =  V_P(n)  \frac{1}{|G|} \sum_{a \in G} \overline{\rho^a} \psi(n+\widehat{a}) = V_P(n)  \psi_\rho(n).
\end{aligned} \end{equation*}
As each $\psi_\rho$ is determined by its values on $\widehat{Q}$, we see that $\Delta + V_P$ acts on each $\psi_\rho$ in the same way that $\Delta + V$ does, that is, as multiplication by $L_V(\rho z)$. As the discrete Fourier transform is unitary, it follows that for each $z \in \TT^d$, $\sigma(L_{V_P}(z^{P/Q})) = \bigcup_{\rho \in U_{P/Q}} \sigma(L_V(\rho z))$, and so we obtain \eqref{eq:finFourier}. 
\end{proof}
As sending $z \to z^{P/Q}$ does not affect the spectral invariants, it follows from \eqref{eq:finFourier} that any nonzero $Q\ZZ$-periodic potential Floquet isospectral to $\bfz$ naturally induces a nonzero $P\ZZ$-periodic potential Floquet isospectral to $\bfz$.

The following proposition shows that we can restrict our search for potentials Floquet isospectral to $\bfz$ to $Q\ZZ$-periodic potentials where $q_i \geq 3$ for all $i$.

\begin{proposition}\label{prop: lifting solutions}
    Suppose $Q=(q_1,\dots, q_{d-1}, 1)$ or $Q = (q_1,\dots, q_{d-1}, 2)$. Then there exists a nonzero $Q\ZZ$-periodic potential Floquet isospectral to $\bfz$ if and only if there is a nonzero $(q_1,\dots,q_{d-1})\ZZ$-periodic potential that is Floquet isospectral to $\bfz$.  
\end{proposition}
\begin{proof}
    Notice that one direction is trivial; in particular, if there is a nonzero $(q_1,\dots,q_{d-1})\ZZ$-periodic potential Floquet isospectral to $\bfz$, then it is clear by equations \eqref{eq:defmatricesgen} (case $q_d=1$) and \eqref{eq:finFourier} that there is a nonzero $Q\ZZ$-periodic potential that is Floquet isospectral to $\bfz$. 

  When $q_d = 1$, the remaining direction follows from absorbing $z_d+z_d^{-1}$ into $\lambda$ in the characteristic matrix. It remains to consider when $q_d =2$.  
  
  Suppose that there is a nonzero $Q\ZZ$-periodic potential Floquet isospectral to $\bfz$. By equation \eqref{eq:defmatricesgen},
\[
           L_{V}=\begin{pmatrix}
            L_{\tilde{V}_1} & I_{\tilde{Q}} + z_d^{-1} I_{\tilde{Q}} \\
            I_{\tilde{Q}} + z_d I_{\tilde{Q}}  & L_{\tilde{V}_2} \\
        \end{pmatrix}.\]

\noindent Thus, $D_V(z_1,\dots,z_{d-1},-1,\lambda) = \det(L_{\tilde{V}_1} -\lambda I) \det(L_{\tilde{V}_2} -\lambda I)$. 

As each $\det(L_{\tilde{V}_1} -\lambda I)$ is irreducible~\cites{fg,flm22} for all potentials $V$, both $\det(L_{\tilde{V}_1} -\lambda I) = \det(L_{\tilde{\bfz}} -\lambda I)$ and $\det(L_{\tilde{V}_2} -\lambda I) = \det(L_{\tilde{\bfz}} -\lambda I)$. That is, both $\tilde{V}_1$ and $\tilde{V}_2$ must be $(q_1,\dots,q_{d-1})\ZZ$-periodic potentials that are Floquet isospectral to $\bfz$. As we assumed that $V$ is nonzero, at least one of $\tilde{V}_1$ or $\tilde{V}_2$ must be nonzero, and thus there must exist a nonzero $(q_1,\dots,q_{d-1})\ZZ$-periodic potential that is isospectral to $\bfz$.
\end{proof}

With Proposition \ref{prop: lifting solutions} established and assuming the results given in Section \ref{sec: large periods}, we can now give a complete proof of Theorem~\ref{Cor:FloquetMain}.

\begin{proof}[Proof of Theorem~\ref{Cor:FloquetMain}]
    Let $Q = (q_1,\dots, q_d) \in \NN^d$. Due to Proposition \ref{prop: lifting solutions}, we are free to restrict our search to periods $Q\ZZ$ such that $q_i \geq 3$ for all $1\leq i \leq d$. Thus, when the periods are small (that is, case (\ref{case 1})) it suffices to show that a nonzero $(3\ZZ)^d$-periodic potential isospectral to $\bfz$ cannot exist. It can be shown that no nonzero $(3\ZZ)^d$-periodic potentials isospectral to $\bfz$  exist via symbolic computation for $d \leq 2$, and consequently this extends to periods $Q$ such that all $q_i \leq 3$ and at most two of the $q_i$ equal $3$. When the periods are large (that is, case (\ref{case 2})), there is at least one $1\leq i \leq d$ such that $q_i>3$. By equation (\ref{eq:finFourier}), it suffices to show there is a nonzero $P\ZZ$-periodic potential isospectral to $\bfz$ where $P = (1,\dots, 1,q_i)$. Equivalently, by Proposition \ref{prop: lifting solutions}, we can show the existence of a nonzero $q_i\ZZ$-periodic potential isospectral to $\bfz$. That is, we have reduced the problem to 1-dimension, which is then proven in Theorem~\ref{thm:SolutionExistsForq>4}.
\end{proof}

Section \ref{sec: large periods} contains a complete proof of the 1-dimensional case of Theorem~\ref{Cor:FloquetMain} (by combining Theorem~\ref{thm:SolutionExistsForq>4} and Remark \ref{remark: 1D small period case}). 

\subsection{The case of large periods}\label{sec: large periods}
We now consider the case where there exists some $q_i > 3$. As shown in the proof of Theorem~\ref{Cor:FloquetMain} in Section \ref{Sec:FloqCor}, this case reduces to checking if there exists a $q\ZZ$-periodic potential isospectral to $\bfz$ where $q > 3$. Taking the characteristic polynomial of the matrix in equation (\ref{eq:defmatricesdim1}), we see that the support of $D_V(z,\lambda)$ is exactly $(\pm 1,0)$ and $(0, i)$ for $i \in [q]$. Consider the system of spectral invariants from equation (\ref{eqn:spectral_invariants}) that determine when $V$ and $\bfz$ are Floquet isospectral:
\begin{equation*}\label{eqn: 0 spectral invariants}
[z^a\lambda^b] D_V(z,\lambda) - [z^a\lambda^b]D_{\bfz}(z,\lambda) = 0, \text{ for all } (a,b).
\end{equation*}

Notice that $[z^a\lambda^b] D_V(z,\lambda) - [z^a\lambda^b]D_{\bfz}(z,\lambda)$ is $0$ except for when $a = 0$ and $b \in [q-1]$. Thus, let $S_i := [\lambda^i] D_V(z,\lambda) - [\lambda^i]D_{\bfz}(z,\lambda)$ be our spectral invariants, and let $E_\Delta = ( S_0, \dots, S_{n-1} )$ be our ideal of spectral invariants.  As $z$ does not appear among the ideal of spectral invariants, we may fix $z$ to $1$. It follows that $E_\Delta$ is the ideal of spectral invariants for the following matrix:

    \begin{equation*} \label{eq:matwithlambdim1}
        M = \begin{pmatrix}
            0 & 1 & 0 & \dots & 0 & 1 \\
            1 & 0 & 1 & 0 & \dots & 0 \\
            0 & 1 & \ddots & \ddots & \ddots & \vdots \\
            \vdots & \ddots & \ddots & \ddots & \ddots & 0\\ 
            0 & 0 & \dots & 1 & 0 & 1 \\
            1 & 0 & \dots & 0 & 1 & 0  \\
        \end{pmatrix}.
    \end{equation*}

We immediately obtain:

\begin{theorem} \label{thm:SolutionExistsForq>4}
    For all $q \geq 4$, the vanishing set of $E_\Delta$ contains non-zero solutions.
\end{theorem}
\begin{proof}
As $E_\Delta = E_M$, the claim follows from Corollary~\ref{cor: adjacency matrix}. 
\end{proof}
\begin{rmk}
    For any fixed $q \geq 4$, the ideal of spectral invariants has a non-trivial solution over any algebraically closed field of characteristic zero or $p > 0$ as long as $p > q$. Moreover, over $\CC$, the number of non-zero solutions is always even (and positive) because it is invariant under complex conjugation and there are no nonzero real solutions since $v_1^2 + \ldots + v_q^2 = e_1^2 - 2 e_2$ lies in the ideal of spectral invariants. When $q$ even, explicit solutions were found in~\cite{flmrp}.
\end{rmk}

\begin{rmk}\label{remark: 1D small period case}
    Note that when $q=1, q=2,$ and $q=3$, the ideal of spectral invariants are either $(e_1(v_1))$, $(e_1(v_1,v_2), e_2(v_1,v_2))$ or $(e_1(v_1,v_2,v_3), e_2(v_1,v_2,v_3), e_3(v_1,v_2,v_3) - e_1(v_1,v_2,v_3))$ respectively, all of which have only the trivial solution at the origin. This completes the proof of Theorem~\ref{Cor:FloquetMain} in the $1$-dimensional case.
\end{rmk}

\subsection{A Comment on the Case of Small Periods.}\label{sec: small periods}
As one can see, Theorem~\ref{Cor:FloquetMain} does not yield a complete classification of when there exist non-zero $Q\ZZ$-periodic potentials isospectral to $0$. This shortcoming is only due to a lack of understanding for the case when $Q\ZZ = (3\ZZ)^d$ for $d \geq 3$. For $d=2$, one can symbolically show either that all elementary symmetric polynomials are in $E_\Delta$, or that the degree of $E_\Delta$ and the degree of the ideal of its tangent cone agree, so $V = \bfz$ is the only solution. We verified both using the Macaulay2~\cite{M2} computer algebra system. 

Dealing with these remaining cases would be desirable, with $(3\ZZ)^3$ as the smallest open case. Unfortunately, for this case, we were unable to construct a dimension reduction trick as done in Proposition~\ref{prop: lifting solutions}. Even obtaining the generators of $E_\Delta$ explicitly in this case does not appear to be computationally feasible, at least not with the methods we attempted on the hardware to which we had access. Interestingly, this is also the smallest case in which it is unknown whether the band edges can admit codimension $1$ Fermi varieties~\cite{fk2}.

\section*{Acknowledgments}

Cobb acknowledges the support of the National Science Foundation Grant DMS-2402199. Faust was partially supported by the National Science Foundation DMS--2052519 grant.  The authors also thank Ilya Kachkovskiy, Wencai Liu, Luke Oeding, and Frank Sottile for various useful correspondences during the preparation of this work. We also thank Frank Sottile for providing the graphic. 
\bibliographystyle{abbrv} % abbrv
	\bibliography{final}

\begin{bibdiv}
\begin{biblist}

\bib{al2024characterizing}{article}{
      author={Al~Ahmadieh, Abeer},
      author={Vinzant, Cynthia},
       title={Characterizing principal minors of symmetric matrices via
  determinantal multiaffine polynomials},
        date={2024},
     journal={J. Algebra},
      volume={638},
       pages={255\ndash 278},
}

\bib{alexander1978additive}{article}{
      author={Alexander, JC},
       title={The additive inverse eigenvalue problem and topological degree},
        date={1978},
     journal={Proc. Amer. Math. Soc.},
      volume={70},
      number={1},
       pages={5\ndash 7},
}

\bib{borg}{article}{
      author={Borg, G\"{o}ran},
       title={Eine {U}mkehrung der {S}turm-{L}iouvilleschen {E}igenwertaufgabe.
  {B}estimmung der {D}ifferentialgleichung durch die {E}igenwerte},
        date={1946},
        ISSN={0001-5962},
     journal={Acta Math.},
      volume={78},
       pages={1\ndash 96},
         url={https://doi.org/10.1007/BF02421600},
      review={\MR{15185}},
}

\bib{Cartwright2009HilbertSchemeOf8Points}{article}{
      author={Cartwright, Dustin~A.},
      author={Erman, Daniel},
      author={Velasco, Mauricio},
      author={Viray, Bianca},
       title={Hilbert schemes of 8 points},
        date={2009},
        ISSN={1937-0652,1944-7833},
     journal={Algebra Number Theory},
      volume={3},
      number={7},
       pages={763\ndash 795},
         url={https://doi.org/10.2140/ant.2009.3.763},
}

\bib{MoodySurvey}{article}{
      author={Chu, Moody~T.},
       title={Inverse eigenvalue problems},
        date={1998},
     journal={SIAM Review},
      volume={40},
      number={1},
       pages={1\ndash 39},
      eprint={https://doi.org/10.1137/S0036144596303984},
         url={https://doi.org/10.1137/S0036144596303984},
}

\bib{EkedahlSkjelnes2014Hilb}{article}{
      author={Ekedahl, Torsten},
      author={Skjelnes, Roy},
       title={Recovering the good component of the {H}ilbert scheme},
        date={2014},
        ISSN={0003-486X,1939-8980},
     journal={Ann. of Math. (2)},
      volume={179},
      number={3},
       pages={805\ndash 841},
         url={https://doi.org/10.4007/annals.2014.179.3.1},
      review={\MR{3171755}},
}

\bib{eskin89}{article}{
      author={Eskin, G.},
       title={Inverse spectral problem for the {S}chr\"{o}dinger equation with
  periodic vector potential},
        date={1989},
        ISSN={0010-3616},
     journal={Comm. Math. Phys.},
      volume={125},
      number={2},
       pages={263\ndash 300},
         url={http://projecteuclid.org/euclid.cmp/1104179467},
      review={\MR{1016872}},
}

\bib{ERT84}{article}{
      author={Eskin, Gregory},
      author={Ralston, James},
      author={Trubowitz, Eugene},
       title={On isospectral periodic potentials in {${\bf R}\sp{n}$}},
        date={1984},
        ISSN={0010-3640},
     journal={Comm. Pure Appl. Math.},
      volume={37},
      number={5},
       pages={647\ndash 676},
         url={https://doi.org/10.1002/cpa.3160370505},
      review={\MR{752594}},
}

\bib{ERTII}{article}{
      author={Eskin, Gregory},
      author={Ralston, James},
      author={Trubowitz, Eugene},
       title={On isospectral periodic potentials in {${\bf R}\sp{n}$}. {II}},
        date={1984},
        ISSN={0010-3640},
     journal={Comm. Pure Appl. Math.},
      volume={37},
      number={6},
       pages={715\ndash 753},
         url={https://doi.org/10.1002/cpa.3160370602},
      review={\MR{762871}},
}

\bib{FGA}{book}{
      author={Fantechi, Barbara},
      author={G\"ottsche, Lothar},
      author={Illusie, Luc},
      author={Kleiman, Steven~L.},
      author={Nitsure, Nitin},
      author={Vistoli, Angelo},
       title={Fundamental algebraic geometry},
      series={Mathematical Surveys and Monographs},
   publisher={American Mathematical Society, Providence, RI},
        date={2005},
      volume={123},
        ISBN={0-8218-3541-6},
         url={https://doi.org/10.1090/surv/123},
        note={Grothendieck's FGA explained},
      review={\MR{2222646}},
}

\bib{Farkas2024Irrational}{article}{
      author={Farkas, Gavril},
      author={Pandharipande, Rahul},
      author={Sammartano, Alessio},
       title={{Irrational components of the Hilbert scheme of points}},
        date={2024},
     journal={arXiv:2504.03282},
      eprint={2405.11997},
}

\bib{fg}{article}{
      author={Faust, M.},
      author={Garcia, J.~Lopez},
       title={Irreducibility of the dispersion polynomial for periodic graphs},
        date={2025},
     journal={SIAM Journal on Applied Algebra and Geometry},
      volume={9},
      number={1},
       pages={83\ndash 107},
      eprint={https://doi.org/10.1137/23M1600256},
         url={https://doi.org/10.1137/23M1600256},
}

\bib{flmrp}{article}{
      author={Faust, Matthew},
      author={Liu, Wencai},
      author={Matos, Rodrigo},
      author={Robinson, Jonah},
      author={Plute, Jenna},
      author={Tao, Yichen},
      author={Tran, Ethan},
      author={Zhuang, Cindy},
       title={Floquet isospectrality of the zero potential for discrete
  periodic {S}chr\"odinger operators},
        date={2024},
        ISSN={0022-2488},
     journal={J. Math. Phys.},
      volume={65},
      number={7},
       pages={Paper No. 073501, 9},
         url={https://doi.org/10.1063/5.0201744},
}

\bib{flm22}{article}{
      author={Fillman, Jake},
      author={Liu, Wencai},
      author={Matos, Rodrigo},
       title={Irreducibility of the {B}loch variety for finite-range
  {S}chr\"{o}dinger operators},
        date={2022},
        ISSN={0022-1236},
     journal={J. Funct. Anal.},
      volume={283},
      number={10},
       pages={Paper No. 109670, 22},
         url={https://doi.org/10.1016/j.jfa.2022.109670},
      review={\MR{4474842}},
}

\bib{fk2}{article}{
      author={Filonov, Nikolay},
      author={Kachkovskiy, Ilya},
       title={On spectral bands of discrete periodic operators},
        date={2024},
        ISSN={0010-3616,1432-0916},
     journal={Comm. Math. Phys.},
      volume={405},
      number={2},
       pages={Paper No. 21, 17},
         url={https://doi.org/10.1007/s00220-023-04891-7},
      review={\MR{4698054}},
}

\bib{Fogarty}{article}{
      author={Fogarty, John},
       title={Algebraic families on an algebraic surface},
        date={1968},
        ISSN={0002-9327,1080-6377},
     journal={Amer. J. Math.},
      volume={90},
       pages={511\ndash 521},
         url={https://doi.org/10.2307/2373541},
      review={\MR{237496}},
}

\bib{friedland1972matrices}{article}{
      author={Friedland, Shmuel},
       title={Matrices with prescribed off-diagonal elements},
        date={1972},
     journal={Israel J. Math.},
      volume={11},
      number={2},
       pages={184\ndash 189},
}

\bib{FRIEDLAND197715}{article}{
      author={Friedland, Shmuel},
       title={Inverse eigenvalue problems},
        date={1977},
        ISSN={0024-3795},
     journal={Linear Algebra Appl.},
      volume={17},
      number={1},
       pages={15\ndash 51},
  url={https://www.sciencedirect.com/science/article/pii/0024379577900398},
}

\bib{gki}{article}{
      author={Gordon, Carolyn~S.},
      author={Kappeler, Thomas},
       title={On isospectral potentials on tori},
        date={1991},
        ISSN={0012-7094},
     journal={Duke Math. J.},
      volume={63},
      number={1},
       pages={217\ndash 233},
         url={https://doi.org/10.1215/S0012-7094-91-06310-6},
      review={\MR{1106944}},
}

\bib{M2}{misc}{
      author={Grayson, Daniel~R.},
      author={Stillman, Michael~E.},
       title={Macaulay2, a software system for research in algebraic geometry},
         how={Available at \url{http://www2.macaulay2.com}},
}

\bib{grinshpan2013norm}{article}{
      author={Grinshpan, Anatolii},
      author={Kaliuzhnyi-Verbovetskyi, Dmitry~S},
      author={Woerdeman, Hugo~J},
       title={Norm-constrained determinantal representations of multivariable
  polynomials},
        date={2013},
     journal={Complex. Anal. Oper. Theory},
      volume={7},
      number={3},
       pages={635\ndash 654},
}

\bib{gui90}{article}{
      author={Guillemin, V.},
       title={Inverse spectral results on two-dimensional tori},
        date={1990},
        ISSN={0894-0347},
     journal={J. Amer. Math. Soc.},
      volume={3},
      number={2},
       pages={375\ndash 387},
         url={https://doi.org/10.2307/1990958},
      review={\MR{1035414}},
}

\bib{GU}{article}{
      author={Guillemin, V.},
      author={Uribe, A.},
       title={Spectral properties of a certain class of complex potentials},
        date={1983},
        ISSN={0002-9947},
     journal={Trans. Amer. Math. Soc.},
      volume={279},
      number={2},
       pages={759\ndash 771},
         url={https://doi.org/10.2307/1999566},
      review={\MR{709582}},
}

\bib{HartshorneConnectedness}{article}{
      author={Hartshorne, Robin},
       title={Connectedness of the {H}ilbert scheme},
        date={1966},
        ISSN={0073-8301,1618-1913},
     journal={Inst. Hautes \'Etudes Sci. Publ. Math.},
      number={29},
       pages={5\ndash 48},
         url={http://www.numdam.org/item?id=PMIHES_1966__29__5_0},
      review={\MR{213368}},
}

\bib{H}{book}{
      author={Hartshorne, Robin},
       title={Algebraic geometry},
      series={Graduate Texts in Mathematics},
   publisher={Springer-Verlag},
        date={1977},
      volume={52},
}

\bib{Burak}{article}{
      author={Hatino\u{g}lu, Burak},
      author={Eakins, Jerik},
      author={Frendreiss, William},
      author={Lamb, Lucille},
      author={Manage, Sithija},
      author={Puente, Alejandra},
       title={Ambarzumian-type problems for discrete {S}chr\"odinger
  operators},
        date={2021},
        ISSN={1661-8254,1661-8262},
     journal={Complex Anal. Oper. Theory},
      volume={15},
      number={8},
       pages={Paper No. 118, 13},
         url={https://doi.org/10.1007/s11785-021-01169-5},
      review={\MR{4328479}},
}

\bib{holtz2007hyperdeterminantal}{article}{
      author={Holtz, Olga},
      author={Sturmfels, Bernd},
       title={Hyperdeterminantal relations among symmetric principal minors},
        date={2007},
     journal={J. Algebra},
      volume={316},
      number={2},
       pages={634\ndash 648},
}

\bib{huang2017symmetrization}{article}{
      author={Huang, Huajun},
      author={Oeding, Luke},
       title={Symmetrization of principal minors and cycle-sums},
        date={2017},
     journal={Linear Mult. Alg.},
      volume={65},
      number={6},
       pages={1194\ndash 1219},
}

\bib{Iarrobino1972Reducibility}{article}{
      author={Iarrobino, A.},
       title={Reducibility of the families of {$0$}-dimensional schemes on a
  variety},
        date={1972},
        ISSN={0020-9910,1432-1297},
     journal={Invent. Math.},
      volume={15},
       pages={72\ndash 77},
         url={https://doi.org/10.1007/BF01418644},
}

\bib{JelisiejewGenericallyNonreduced}{article}{
      author={Jelisiejew, J.},
       title={Generically nonreduced components of {H}ilbert schemes on
  fourfolds},
        date={2024},
        ISSN={0373-1243,2704-999X},
     journal={Rend. Semin. Mat. Univ. Politec. Torino},
      volume={82},
      number={1},
       pages={171\ndash 184},
      review={\MR{4856857}},
}

\bib{JelisiejewPathologies}{article}{
      author={Jelisiejew, Joachim},
       title={Pathologies on the {H}ilbert scheme of points},
        date={2020},
        ISSN={0020-9910,1432-1297},
     journal={Invent. Math.},
      volume={220},
      number={2},
       pages={581\ndash 610},
         url={https://doi.org/10.1007/s00222-019-00939-5},
      review={\MR{4081138}},
}

\bib{JelisiejewComponents}{article}{
      author={Jelisiejew, Joachim},
      author={\v{S}ivic, Klemen},
       title={Components and singularities of {Q}uot schemes and varieties of
  commuting matrices},
        date={2022},
        ISSN={0075-4102,1435-5345},
     journal={J. Reine Angew. Math.},
      volume={788},
       pages={129\ndash 187},
         url={https://doi.org/10.1515/crelle-2022-0018},
      review={\MR{4445543}},
}

\bib{kapi}{article}{
      author={Kappeler, Thomas},
       title={On isospectral periodic potentials on a discrete lattice. {I}},
        date={1988},
        ISSN={0012-7094},
     journal={Duke Math. J.},
      volume={57},
      number={1},
       pages={135\ndash 150},
         url={https://doi.org/10.1215/S0012-7094-88-05705-5},
      review={\MR{952228}},
}

\bib{kapii}{article}{
      author={Kappeler, Thomas},
       title={On isospectral potentials on a discrete lattice. {II}},
        date={1988},
        ISSN={0196-8858},
     journal={Adv. in Appl. Math.},
      volume={9},
      number={4},
       pages={428\ndash 438},
         url={https://doi.org/10.1016/0196-8858(88)90021-8},
      review={\MR{968676}},
}

\bib{kapiii}{article}{
      author={Kappeler, Thomas},
       title={Isospectral potentials on a discrete lattice. {III}},
        date={1989},
        ISSN={0002-9947},
     journal={Trans. Amer. Math. Soc.},
      volume={314},
      number={2},
       pages={815\ndash 824},
         url={https://doi.org/10.2307/2001410},
      review={\MR{961624}},
}

\bib{Kittel}{book}{
      author={Kittel, C.},
      author={McEuen, P.},
       title={Introduction to solid state physics},
   publisher={John Wiley \& Sons},
        date={2018},
}

\bib{KuchBook}{book}{
      author={Kuchment, Peter},
       title={Floquet theory for partial differential equations},
      series={Operator Theory: Advances and Applications},
   publisher={Birkh\"{a}user Verlag, Basel},
        date={1993},
      volume={60},
         url={https://doi.org/10.1007/978-3-0348-8573-7},
}

\bib{ksurvey}{article}{
      author={Kuchment, Peter},
       title={An overview of periodic elliptic operators},
        date={2016},
        ISSN={0273-0979},
     journal={Bull. Amer. Math. Soc. (N.S.)},
      volume={53},
      number={3},
       pages={343\ndash 414},
         url={https://doi.org/10.1090/bull/1528},
      review={\MR{3501794}},
}

\bib{kuchment2023analytic}{article}{
      author={Kuchment, Peter},
       title={Analytic and algebraic properties of dispersion relations
  ({B}loch varieties) and {F}ermi surfaces. {W}hat is known and unknown},
        date={2023},
        ISSN={0022-2488},
     journal={J. Math. Phys.},
      volume={64},
      number={11},
       pages={Paper No. 113504, 11},
         url={https://doi.org/10.1063/5.0152990},
      review={\MR{4667406}},
}

\bib{liujmp22}{article}{
      author={Liu, Wencai},
       title={Topics on {F}ermi varieties of discrete periodic
  {S}chr\"{o}dinger operators},
        date={2022},
        ISSN={0022-2488},
     journal={J. Math. Phys.},
      volume={63},
      number={2},
       pages={Paper No. 023503, 13},
         url={https://doi.org/10.1063/5.0078287},
      review={\MR{4376993}},
}

\bib{liujde}{article}{
      author={Liu, Wencai},
       title={Floquet isospectrality for periodic graph operators},
        date={2023},
        ISSN={0022-0396},
     journal={J. Differential Equations},
      volume={374},
       pages={642\ndash 653},
      review={\MR{4632070}},
}

\bib{liu2021fermi}{article}{
      author={Liu, Wencai},
       title={Fermi isospectrality for discrete periodic {S}chr\"{o}dinger
  operators},
        date={2024},
        ISSN={0010-3640},
     journal={Comm. Pure Appl. Math.},
      volume={77},
      number={2},
       pages={1126\ndash 1146},
      review={\MR{4673878}},
}

\bib{MT76}{article}{
      author={McKean, H.~P.},
      author={Trubowitz, E.},
       title={Hill's operator and hyperelliptic function theory in the presence
  of infinitely many branch points},
        date={1976},
        ISSN={0010-3640},
     journal={Comm. Pure Appl. Math.},
      volume={29},
      number={2},
       pages={143\ndash 226},
         url={https://doi.org/10.1002/cpa.3160290203},
      review={\MR{427731}},
}

\bib{VPapan}{article}{
      author={Papanicolaou, Vassilis~G.},
       title={Periodic {J}acobi operators with complex coefficients},
        date={2021},
        ISSN={1664-039X,1664-0403},
     journal={J. Spectr. Theory},
      volume={11},
      number={2},
       pages={781\ndash 819},
         url={https://doi.org/10.4171/jst/357},
      review={\MR{4293492}},
}

\bib{saburova}{article}{
      author={Saburova, Natalia},
       title={Spectral invariants for discrete schr\"odinger operators on
  periodic graphs},
        date={2025},
     journal={arXiv:2504.03282},
}

\bib{shipman2025algebraic}{article}{
      author={Shipman, Stephen~P},
      author={Sottile, Frank},
       title={Algebraic aspects of periodic graph operators},
        date={2025},
     journal={arXiv:2502.03659},
}

\bib{wa}{article}{
      author={Waters, Alden},
       title={Isospectral periodic torii in dimension 2},
        date={2015},
        ISSN={0294-1449},
     journal={Ann. Inst. H. Poincar\'{e} Anal. Non Lin\'{e}aire},
      volume={32},
      number={6},
       pages={1173\ndash 1188},
         url={https://doi.org/10.1016/j.anihpc.2014.06.001},
      review={\MR{3425258}},
}

\end{biblist}
\end{bibdiv}
	
\end{document}